\documentclass[11pt]{article}
\usepackage{amssymb}
\usepackage{amsmath}
\usepackage{graphicx}
\usepackage{setspace}
\begin{document}

\newtheorem{theorem}{Theorem}[section]
\newtheorem{tha}{Theorem}
\newtheorem{conjecture}[theorem]{Conjecture}
\newtheorem{corollary}[theorem]{Corollary}
\newtheorem{lemma}[theorem]{Lemma}
\newtheorem{claim}[theorem]{Claim}
\newtheorem{proposition}[theorem]{Proposition}
\newtheorem{construction}[theorem]{Construction}
\newtheorem{definition}[theorem]{Definition}
\newtheorem{question}[theorem]{Question}
\newtheorem{problem}[theorem]{Problem}
\newtheorem{remark}[theorem]{Remark}
\newtheorem{observation}[theorem]{Observation}

\newcommand{\ex}{{\mathrm{ex}}}

\newcommand{\EX}{{\mathrm{EX}}}

\newcommand{\AR}{{\mathrm{AR}}}

\def\endproofbox{\hskip 1.3em\hfill\rule{6pt}{6pt}}
\newenvironment{proof}%
{%
\noindent{\it Proof.}
}%
{%
 \quad\hfill\endproofbox\vspace*{2ex}
}
\def\qed{\hskip 1.3em\hfill\rule{6pt}{6pt}}
\def\ce#1{\lceil #1 \rceil}
\def\fl#1{\lfloor #1 \rfloor}
\def\lr{\longrightarrow}
\def\e{\varepsilon}
\def\ex{{\rm\bf ex}}
\def\cB{{\cal B}}
\def\cC{{\cal C}}
    \def\cD{{\cal D}}
\def\cF{{\cal F}}
\def\cG{{\cal G}}
\def\cH{{\cal H}}
\def\ck{{\cal K}}
\def\cI{{\cal I}}
\def\cJ{{\cal J}}
\def\cL{{\cal L}}
\def\cM{{\cal M}}
\def\cP{{\cal P}}
\def\cQ{{\cal Q}}
\def\cS{{\cal S}}

\def\pr{{\rm Pr}}
\def\exp{{\rm  exp}}

\def\bkl{{\cal B}^{(k)}_\ell}
\def\cmkt{{\cal M}^{(k)}_{t+1}}
\def\cpkl{{\cal P}^{(k)}_\ell}
\def\cckl{{\cal C}^{(k)}_\ell}
\def\pkl{\mathbb{P}^{(k)}_\ell}
\def\ckl{\mathbb{C}^{(k)}_\ell}

\def\mC{{\cal C}}

\def\imp{\Longrightarrow}
\def\1e{\frac{1}{\e}\log \frac{1}{\e}}
\def\ne{n^{\e}}
\def\rad{ {\rm \, rad}}
\def\equ{\Longleftrightarrow}
\def\pkl{\mathbb{P}^{(k)}_\ell}

\def\c2ell{\mathbb{C}^{(2)}_\ell}
\def\codd{\mathbb{C}^{(k)}_{2t+1}}
\def\ceven{\mathbb{C}^{(k)}_{2t+2}}
\def\podd{\mathbb{P}^{(k)}_{2t+1}}
\def\peven{\mathbb{P}^{(k)}_{2t+2}}
\def\TT{{\mathbb T}}
\def\bbT{{\mathbb T}}
\def\cE{{\mathcal E}}
\def\e{\epsilon}

\def\mE{\mathbb{E}}

\def\wt{\widetilde}
\def\wh{\widehat}
\voffset=-0.5in
	
\setstretch{1.1}
\pagestyle{myheadings}
\markright{{\small \sc  Collier-Cartaino, Graber, Jiang:}
  {\it\small Linear Tur\'an numbers of  linear cycles}}

\title{\huge\bf 
Linear Tur\'an numbers of $r$-uniform linear cycles and related Ramsey numbers}

\author{
Clayton Collier-Cartaino\thanks{Dept. of Mathematics, Miami University, Oxford, OH 45056, USA. E-mail:colliec@miamioh.edu
}\quad\quad
Nathan Graber\thanks{Dept. of Mathematics, Miami University, Oxford, OH 45056,
USA. E-mail: grabernt@miamioh.edu} \quad \quad
Tao Jiang\thanks{Dept. of Mathematics, Miami University, Oxford,
OH 45056, USA. E-mail: jiangt@miamioh.edu. Research supported in part by
Simons Foundation Collaboration Grant \#282906.
    %{\rm\small {\jobname}.tex,}\newline ${}$ \hfill}
} }

\date{}

\maketitle
\begin{abstract}
An $r$-uniform hypergraph is called an {\it $r$-graph}.
A hypergraph is {\it linear} if every two edges intersect in at most one vertex.
Given a linear $r$-graph $H$ and a positive integer $n$, the {\it linear Tur\'an number}
$ex_L(n,H)$ is the maximum number of edges in a linear $r$-graph $G$ that
does not contain $H$ as a subgraph.  For each $\ell\geq 3$, let $C^r_\ell$ denote the $r$-uniform linear cycle of length $\ell$, which is an $r$-graph with edges $e_1,\ldots, e_\ell$ such
that $\forall i\in [\ell-1]$, $|e_i\cap e_{i+1}|=1$, $|e_\ell\cap e_1|=1$ and $e_i\cap e_j=\emptyset$ for all other pairs $\{i,j\}, i\neq j$.
For all $r\geq 3$ and $\ell\geq 3$, 
we show that there exist positive constants $c_{m,r}$ and $c'_{m,r}$, 
depending only $m$ and $r$, such
that  $ex_L(n,C^r_{2m})\leq c_{m,r} n^{1+\frac{1}{m}}$ and $ex_L(n,C^r_{2m+1})\leq c'_{m,r}
n^{1+\frac{1}{m}}$. 
This answers a question of Kostochka, Mubayi, and Verstra\"ete \cite{KMV-personal}. For even cycles, our result extends the result of Bondy and Simonovits \cite{BS} 
on the Tur\'an numbers of even cycles to linear hypergraphs. 

 Using our results on linear Tur\'an numbers we also obtain bounds on the cycle-complete
hypergraph Ramsey numbers.  We show that there are positive constants $a_{m,r}$ and $b_{m,r}$, depending only on $m$ and $r$, such that 
$R(C^r_{2m}, K^r_t)\leq a_{m,r} (\frac{t}{\ln t})^\frac{m}{m-1}$ and 
$R(C^r_{2m+1}, K^r_t)\leq b_{m,r} t^\frac{m}{m-1}$.
\end{abstract}

\section{Introduction}

A {\it hypergraph} $H=(V,E)$ consists of a set $V$ of vertices and a set $E$ of edges, where each edge is a subset of $V$. 
%If $V$ has $n$ vertices, often times it is convenient to just let $V=[n]$. 
If all the edges of $H$ have size $r$, then $H$ is said to be
{\it $r$-uniform} and will be called an $r$-graph for brevity. 
The complete $r$-graph on $n$ vertices will be denoted by $K^{r}_n$.
A hypergraph $H$ is {\it linear} if $\forall e,e'\in E(H), |e\cap e'|\leq 1$.
Given a family $\cH$ of $r$-graphs, the {\it Tur\'an number} of $\cH$ for a given
positive integer $n$,  denoted by $ex(n,\cH)$, is the maximum number of edges of an
$r$-graph on $n$ vertices that does not contain any member of $\cH$ as a subgraph.
If $\cH$ is a family of linear $r$-graphs, then we define, for a given positive integer $n$,
the {\it linear Tur\'an number} of $\cH$
to be the maximum number of edges of a linear $r$-graph on $n$ vertices that does not contain
any member of $\cH$ as a subgraph, and denote it by $ex_L(n,\cH)$. When $\cH$ consists of a single graph $H$, we write $ex(n,H)$ and $ex_L(n,H)$ for $ex(n,\cH)$ and $ex_L(n,\cH)$, respectively.

A linear cycle of length $\ell$ is a hypergraph with edges $e_1,\ldots, e_\ell$ such that
$\forall i\in [\ell-1]$, $|e_i\cap e_{i+1}|=1$, $|e_\ell\cap e_1|=1$ and $e_i\cap e_j=\emptyset$ for all other pairs $\{i,j\}, i\neq j$. We denote an $r$-uniform linear cycle of length $\ell$ by
$C^{r}_\ell$. In particular, $2$-uniform linear cycles are just the usual graph cycles.
The Tur\'an problem for graph cycles has been much studied. For odd cycles, the answer is
$\fl{\frac{n^2}{4}}$ for all sufficiently large $n$, with equality achieved by
a balanced complete bipartite graph on $n$ vertices. The problem for even cycles 
remains unresolved  except for $C_4$ \cite{Furedi-C4}. A general upper bound of $ex(n,C_{2m})\leq \gamma_m n^{1+\frac{1}{m}}$ for
some positive constant $\gamma_m$ was asserted by Erd\H{o}s (unpublished). The first
published proof was obtained by Bondy and Simonovits \cite{BS}, who showed that $ex(n,C_{2m})
\leq 20mn^{1+\frac{1}{m}}$ for all sufficiently large $n$. This was improved by Verstra\"ete
\cite{Verstraete-cycles} to $8(m-1)n^{1+\frac{1}{m}}$ and by Pikhurko \cite{Oleg} to
$(m-1)n^{1+\frac{1}{m}}$. Very recently, Bukh and Jiang \cite{BJ}  improved the upper bound to $80\sqrt{m\log m}\cdot n^{1+\frac{1}{m}}+10m^2n$ for all $n\geq (2m)^{8m^2}$.
For $m=2,3,5$, constructions of $C_{2m}$-free $n$-vertex graphs with $\Omega(n^{1+\frac{1}{m}})$ edges
are known (see \cite{Furedi-Simonovits}). Thus $ex(n,C_{2m})=\Theta(n^{1+\frac{1}{m}})$,
for $m\in \{2,3,5\}$. However, the order of magnitude 
of $ex(n,C_{2m})$ remains undetermined for all $m\not\in \{2,3,5\}$.

The Tur\'an problem for hypergraph cycles has also been explored.
There are several different notions of hypergraph cycles. A 
hypergraph $H$ is a {\it Berge cycle} of length $\ell$ if it consists of
$\ell$ distinct edges $e_1,\ldots, e_\ell$ such that there exists a list
of distinct vertices $x_1,\ldots, x_\ell$ satisfying that $\forall i\in [\ell-1]$
$e_i$ contains both $x_i$ and $x_{i+1}$ and that $e_\ell$ contains both $x_\ell$ and $x_1$.
Note that a $2$-uniform Berge cycle of length $\ell$ is just the usual graph cycle
of length $\ell$. For $r\geq 3$, however, $r$-uniform Berge cycles are not unique as
there are no constraints on how the $e_i$'s intersect outside $\{x_1,\ldots, x_\ell\}$.
Let $\cB^r_\ell$ denote the family of $r$-graphs that are Berge cycles of length $\ell$.
Gy\H{o}ri and Lemons \cite{GL-3-uniform, GL-all} showed that for all $r\geq 3, \ell\geq 3$,
there exists a positive constant $\beta_{r,\ell}$, depending on $r$ and $\ell$ such
that $ex(n,\cB^r_\ell)\leq \beta_{r,\ell} n^{1+\frac{1}{\fl{\ell/2}}}$.
Another notion of hypergraph cycles that has been actively investigated recently is
that of a linear cycle defined earlier. For fixed $r,\ell$, the $r$-uniform linear cycle $C^r_\ell$ 
of length $\ell$ is unique up to isomorphism.
We can also describe an $r$-uniform linear cycle using the notion of expansions.
Given a $2$-graph $G$, the {\it $r$-expansion} $G^{(r)}$ is the $r$-graph
obtained from $G$ 
by enlarging each edge of $G$ into an $r$-set using $r-2$ new vertices, called
{\it expansion vertices}, such that for different edges of $G$ we use disjoint sets of expansion
vertices. So an $r$-uniform linear cycle of length $\ell$ is precisely the $r$-expansion of a cycle of length $\ell$. F\"uredi and Jiang \cite{FJ} determined for all $r\geq 5,\ell\geq 3$ and sufficiently large $n$ the exact value of $ex(n,C^r_\ell)$, showing
that $ex(n,C^r_{2m+1})=\binom{n}{r}-\binom{n-m}{r}$ and $ex(n,C^r_{2m})=\binom{n}{r}
-\binom{n-m+1}{r}+\binom{n-m-1}{2}$, respectively. Kostochka, Mubayi, and Verstra\"ete 
\cite{KMV-shadows} 
have subsequently showed that the same holds for all $r\geq 3$, $\ell\geq 3$, and
sufficiently large $n$. In this paper, we study the linear Tur\'an number of $C^r_\ell$. 

Determining $ex_L(n,C^3_3)$ is equivalent to the famous $(6,3)$-problem,which is a special case of an old and general extremal problem of  Brown, Erd\H{o}s,
and S\'os \cite{BES}. The Brown-Erd\H{o}s-S\'os problem asks to determine the function 
$f_r(n,v,e)$, which  denotes the maximum number of edges in an $r$-graph on
$n$ vertices in which no $v$ vertices spans $e$ or more edges. The problem of
estimating $f_3(n,6,3)$ is known as the $(6,3)$-problem. 
It is easy to see that $ex_L(n,C^3_3)=f_3(n,6,3)$.  In one of the classical
results in extremal combinatorics, Ruzsa and Szemer\'edi \cite{RS} showed that for some constant $c>0$,
\begin{equation}\label{RS}
n^{2-c\sqrt{\log n}}< f_3(n,6,3)=o(n^2).
\end{equation}

\begin{proposition} 
$n^{2-c\sqrt{\log n}}< ex_L(n,C^3_3)=o(n^2).$
\end{proposition}
Let us reiterate the following connection in the literature between $f_3(n,6,3)$ and the function $r_3(n)$, which denotes the largest size of a set of integers in $[n]$ not containing a
$3$-term arithmetic progression.
 Given $n$, let $N=\fl{\frac{n}{6}}$ and let $A$ be a subset of size $r_3(N)$ that contains no $3$-term arithmetic progression. Let $X, Y, Z$ be disjoint sets with $X=[N], Y=[2N], Z=[3N]$, respectively. 
The $3$-partite $3$-graph $H=\{\{x,y,z\}: x\in X, y\in Y, z\in Z, \exists
a\in A \, \,  y=x+a, z=x+2a\}$ satisfies that no six points spanns three or more edges and
$|H|=Nr_3(N)$. Hence $f(n,6,3)\geq \fl{\frac{n}{6}} \cdot r_3(\fl{\frac{n}{6}})$.

The upper bound in \eqref{RS}, established with a short proof using regularity lemma in \cite{RS},  implies Roth's theorem \cite{Roth} that $r_3(n)=o(n)$.
Conversely, the lower bound in \eqref{RS} was established using Behrend's \cite{Beh} construction of large subsets of $[n]$  not containing a $3$-term arithmetic progression.
Behrend's construction has size $\Omega(n^{1-c'\sqrt{\log n}})$, for some constant $c'>0$.
Ever since Roth's theorem \cite{Roth}, the problem of estimating $r_3(n)$ has drawn much interest. The best current bounds are as follows: for some constant $c>0$
\begin{equation}
\frac{n}{e^{c\sqrt{\log n}}}\leq r_3(n)\leq \frac{n}{(\log n)^{1-o(1)}}.
\end{equation}
Back to the linear cycle problem, observe that the graph $H$ constructed above is linear and contains no linear triangle.
Using a construction similar to $H$ and so-called $2$-fold Sidon sets, Lazebnik and Verstra\"ete 
\cite{LV} constructed  linear $3$-graphs with girth $5$ and $\Omega(n^{3/2})$ edges. 
On the other hand,  it is not hard to show that $ex_L(n,C^3_4)=O(n^{3/2})$. 
Hence $ex_L(n,C^3_4)=\Theta(n^{3/2})$.
Kostochka, Mubayi, and Verstra\"ete \cite{KMV-personal} obtained the following bounds 
for $ex_L(n,C^3_5)$.
\begin{theorem}\label{C5} {\bf \cite{KMV-personal}}
There are constants $a,b>0$ such that $an^{3/2}<ex_L(n,C^3_5)<bn^{3/2}$.
\end{theorem}

No lower or upper bounds on $ex_L(n,C^r_\ell)$ were formerly known for $\ell\not \in \{3,4,5\}$. 
Kostochka, Mubayi, and Verstra\"ete \cite{KMV-personal} asked if for all $r\geq 3, \ell\geq 3$, $ex_L(n,C^r_\ell)=O(n^{1+\frac{1}{\fl{\ell/2}}})$ holds. 
We answer their question in the affirmative in our main theorem below.

\begin{theorem} {\bf (Main theorem)} \label{main}
For all $r,\ell\geq 3$, there exists a constant $\alpha_{r,\ell}>0$, depending on 
$r$ and $\ell$,  such that 
$$ex_L(n,C^r_\ell)\leq \alpha_{r,\ell} n^{1+\frac{1}{\fl{\ell/2}}}.$$
\end{theorem} 
Another motivation for our study of $ex_L(n,C^r_\ell)$ comes from
the study of the hypergraph Ramsey number $R(C^r_\ell, K^r_t)$ 
of a linear cycle versus a complete graph.
Such a study was initiated by Kostochka, Mubayi, and Verstra\"ete in \cite{KMV}.
Using Theorem \ref{main} and other tools, we obtain nontrivial upper bounds
on $R(C^r_\ell, K^r_t)$. Since our main emphasis of the paper is on 
the linear Tur\'an problem of linear cycles, 
we delay the discussion of the related Ramsey numbers to Section \ref{cycle-complete}.

The rest of the paper is organized as follows. Section 2 contains some notation and terminology.
Section 3 contains some lemmas needed for our main theorem. Section 4 contains the proof
of the main theorem for even cycles. Section 5 contains some additional tools needed for
the proof for odd cycles. Section 6 contains the proof of the main theorem for odd cycles (which is much more involved than for even cycles). Section 7 contains results on cycle-complete
hypergraph Ramsey numbers. Section 8 contains concluding remarks, including some discussion
on the lower bounds on $ex_L(n,C^r_\ell)$.

Our main method has roots in \cite{FS} and \cite{Jiang-Seiver}, but requires a substantial innovation for the odd cycle case. The new ideas used there could potentially have applications in other problems. 

%%%%%%%%%%%%%%%%%%%%%%%%%%%%%%%%%%%%%%%%%%%%%%%%%%%%%%%%%%%%%

\section{Notation and terminology} \label{notation}

\subsection{Degrees, neighborhoods, link graphs}

Let $G$ be a hypergraph. Given a set $S\subseteq V(G)$, we define the {\it degree}
of $S$ in $G$, denoted by $d_G(S)$, to be the number of edges of $G$ that contain $S$.
Given a vertex $x\in V(G)$, we define the {\it link graph} $\cL_G(x)$ of $x$ in $\cG$ as
$\cL_G(x)=\{e\setminus \{x\}: x\in e\in G\}$. Hence if $G$ is an $r$-graph, then
$\cL_G(x)$ is an $(r-1)$-graph. The {\it neigborhood} $N_G(x)$ of $x$ in $G$ is defined
as $N_G(x)=\{u\in V(G): d_G(\{u,x\})\geq 1\}$. When the context is clear, we will drop the subscripts in the above definitions.

\medskip

\subsection{$r$-expansions}

Let $k,r$ be position integers where $r>k\geq 2$. Given a $k$-graph $H$ 
the {\it $r$-expansion} of $H$, denoted by $H^{(r)}$, is the $r$-graph obtained from $H$
enlarging each edge $e$ of $H$ into an $r$-set through a set $A_e$ of $r-k$ new vertices, called {\it expansion vertices}, such that whenever $e\neq e'$ we have $A_e\cap A_{e'}=\emptyset$. So, for instance, the $r$-expansion of a $2$-uniform $\ell$-cycle
is precisely an $r$-uniform linear $\ell$-cycle. We will call $H$ the {\it skeleton} of
$H^{(r)}$.

\subsection{Leveled linear trees}

Given a $2$-uniform tree $T$ rooted at $w$,  $\forall i\geq 0$, let $L_i=\{x: dist_T(w,x)=i\}$.
We call $L_i$ {\it level $i$}. The {\it height} of $T$ is the maximum $i$ for which $L_i\neq \emptyset$. For each $x\in V(T)$, let $T_x$ denote the subtree of $T$ under $x$.
Let $H=T^{(r)}$. Let $f$ be a specific mapping of $T$ to $H$ that maps each $e\in T$
to $e\cup A(e)$ where $A(e)$ is the set of expansion vertices for $e$.
We call $H$ a {\it leveled linear $r$-tree} rooted at $w$ and
will refer to the $L_i$'s as {\it  levels} of $H$. The height of $H$ is defined
to be the height of $T$. If $x$ is a vertex in $L_i$ for some $i$, then
the {\it subtree under $x$} in $H$, denoted by $H_x$, is the image under $f$
of $T_x$ in $H$. %So the edges of $H_x$ are those in $H$ corresponding to edges of $T_x$.

\subsection{Proper, rainbow, strongly proper, strongly rainbow edge-colorings}

Let $c$ be an edge-coloring of a $2$-graph $G$ using natural numbers.
We say that $c$ is {\it proper} if whenever $e$ and $e'$ are incident edges in $G$, $c(e)\neq c(e')$ and we say that $c$ is {\it rainbow} if for every two different edges $e$ and $e'$
in $G$ we have $c(e)\neq c(e')$.  Let $\phi$ be an edge-coloring of a $2$-graph $G$
using $p$-subsets of some ground set $S$. We say that $\phi$ is {\it strongly proper}
if whenever $e$ and $e'$ are incident edges in $G$, $c(e)\cap c(e')=\emptyset$.
We say that $\phi$ is {\it strongly rainbow} if for every two different edges $e$ and
$e'$ in $G$ we have $c(e)\cap c(e')=\emptyset$.

\subsection{Default edge-colorings}

Let $G$ be an $r$-graph.  The $2$-shadow $\partial_2(G)$ of $G$ is the $2$-graph
consisting of all pairs $(a,b)$ that are contained in some edge of $G$. If $G$ is linear then
each edge in $\partial_2(G)$ is contained in a unique edge of $G$. We define the {\it default
edge-coloring} $\phi$ of $\partial_2(G)$ by letting  $\phi(\{a,b\})=e\setminus \{a,b\}$, where
$e$ is the unique edge of $G$ containing $\{a,b\}$.  So $\phi$ is a coloring whose colors are $(r-2)$-sets. If $B\subseteq \partial_2(G)$ then the default edge-coloring of $B$ is defined to be $\phi$ restricted to $B$.

%%%%%%%%%%%%%%%%%%%%%%%%%%%%%%%%%%%%%%%%%%%%%%%%%%%%%%%%%%%%%

\section{Lemmas}

In this section, we prove some lemmas that will be needed in our main proofs.
Let $H$ be a hypergraph.
A {\it vertex cover} of $H$ is a set $Q$ of vertices in $H$ that contains
at least one vertex of each edge of $H$. 
A {\it cross-cut} of $H$ is a set $S$ of vertices in $H$ that contains exactly one
vertex of each edge of $H$.  A {\it matching} in $H$ is a set of
pairwise disjoint edges. The {\it size} of a matching is the number of edges in it.
 
\begin{lemma} \label{matching-cover}
Let $H$ be a $k$-graph, where $k\geq 2$. Let $Q$ be a minimum vertex cover of $H$.
Then $H$ contains a matching of size at least $|Q|/k$.
\end{lemma}
\begin{proof}
Let $M$ be a maximum matching in $H$ and $S$ the set of vertices contained
in edges of $M$. If some edge $e$ of $H$ contains no vertex in $S$ then $M\cup e$
is a larger matching in $H$ than $M$, contradicting our choice of $M$. So
$S$ is a vertex cover of $H$ of size $k|M|$. Since $Q$ is a minimum vertex cover of $H$,
we have $k|M|\geq |Q|$. Thus, $|M|\geq |Q|/k$.
\end{proof}

\begin{lemma} \label{cross-cut}
Let $H$ be a $k$-graph, where $k\geq 2$. Let $S$ be a vertex cover of $H$.
Then there exist a subgraph $H'\subseteq H$ and a subset $S'\subseteq S$
such that $|H'|\geq \frac{k}{2^k}|H|$ and that $S'$ is a cross-cut of $H'$.
\end{lemma}
\begin{proof}
Let $\wt{S}$ be a random subset of $S$ with each vertex of $S$ chosen independently with
probability $\frac{1}{2}$. For each $e\in H$, the probability that
exactly one vertex of $e\cap S$ is included in $\wt{S}$ is $\frac{|e\cap S|}{2^{|e\cap S|}}\geq \frac{k}{2^k}$.
So the expected number of edges $e$ that intersects $\wt{S}$ in exactly one vertex is at least
$\frac{k}{2^k}|H|$. Thus, there exists a subset $\wt{S}$ of $S$ such that at least $\frac{k}{2^k}|H|$ edges intersect $S'$ in exactly one vertex. Let $H'$ denote the subgraph of $H$ consisting of these edges and $S'=\wt{S}\cap V(H')$. The claim follows.
\end{proof}

\begin{lemma} \label{default-coloring}
Let $r\geq 3$. Let $G$ be a linear $r$-graph. Let $B\subseteq \partial_2(G)$ satisfy that
each edge of $G$ contains at most one edge of $B$.
Let $\phi$ be the default edge-coloring of $B$. Then $\phi$ is strongly proper.
\end{lemma}
\begin{proof}
Let $f_1,f_2$ be two edges in $B$ that share a vertex, say $u$. 
Let $e_1,e_2$ be the unique edges of $G$ containing $f_1,f_2$ respectively. By our assumption,
$e_1\neq e_2$.
If $e_1\setminus f_1$ and $e_2\setminus f_2$ share a vertex $v$, then $e_1,e_2$ both contain $\{u,v\}$, contradicting $G$ being linear.
Thus $\phi(\{a,b\})\cap\phi(\{a,c\})=\emptyset$.
\end{proof}

\begin{lemma} \label{rainbow-path}
Let $k,\ell,s$ be positive integers, where $k\geq 2$.
Let $G$ be a $2$-graph with minimum degree at least $(k+1)\ell+s$.
Let $\phi$ be a strongly proper edge-coloring of $G$ using $k$-subsets of some set $S$. 
Let $x \in V(G)$ and $S_0\subseteq S$ with $|S_0|\leq s$. 
Then there exists a path $P$ in $G$ of length $\ell$
starting at $x$ such that (1) $P$ 
is strongly rainbow under $\phi$ and (2) $\forall f\in V(P)$,
$\phi(f)\cap S_0=\emptyset$.
\end{lemma}
\begin{proof}
We use induction on $\ell$. For the basis step, let $\ell=1$. By our assumption, there
are at least $k+s+1$ edges of $G$ incident to $x$. Since $\phi$ is strongly proper,
the colors used on these edges are pairwise disjoint $k$-sets. Certainly one of them is
completely disjoint from $S_0$. Let $e$ be an edge incident to $x$ with $\phi(e)\cap S_0=\emptyset$. The claim holds with $P=e$.
For the induction step, let $\ell>1$.
By induction hypothesis, there is a path $P$ of length $\ell-1$ starting at $x$ such
that  (1) $P$ is strongly rainbow under $\phi$ and (2) $\forall f\in P, \phi(f)\cap S_0=\emptyset$.  Let $S_1=\bigcup_{f\in P} \phi(f)$. Then $|S_1|=k(\ell-1)$.
Let $y$ denote the other endpoint of $P$.
There at least $(k+1)\ell+s$ edges incident to $y$. More than $k\ell+s$ of these join $y$ 
to vertices outside $P$. Since  $\phi$
is strongly proper, the colors on these edges are pairwise disjoint $k$-subsets of $S$. 
Since $k\ell+s>k(\ell-1)+s=|S_0\cup S_1|$, for one of these edges $e$, we have $\phi(e)\cap (S_0\cup S_1)=\emptyset$.
Now, $P\cup e$ is a path of length $\ell$ in $G$
starting at $x$ such that (1) $P\cup e$ is strongly rainbow under $\phi$
and (2) $\forall f\in P\cup e, \phi(f)\cap S_0=\emptyset$.
\end{proof}

\begin{lemma} \label{subgraph}
Let $G$ be a graph with average degree $d$. There exists a subgraph $G'\subseteq G$ such
that $\delta(G')\geq \frac{d}{4}$ and that $|G'|\geq \frac{|G|}{2}$.
\end{lemma}
\begin{proof}
Suppose $G$ has $n$ vertices. Iteratively remove a vertex (and its incident edges) whose degree in the remaining subgraph is less than $\frac{d}{4}$ until no such vertex exists. Let $G'$ denote the remaining subgraph. In the process, fewer than $\frac{nd}{4}\leq \frac{1}{2}|G|$ edges have been removed.
So $|G'|\geq \frac{|G|}{2}$. In particular, $G'$ is nonempty. By our rule, we also have
$\delta(G')\geq \frac{d}{4}$.
\end{proof}

\begin{lemma} \label{min-degree}
Let $G$ be an $r$-graph with average degree $d$. Then $G$ contains
a subgraph $G'$ with $\delta(G')\geq d/r$.
\end{lemma}
\begin{proof}
Suppose $G$ has $n$ vertices.
Starting with $G$, whenever some 
vertex has at most $d/r$ in the remaining graph, we remove this vertex
and  all the edges in the remaining graph that contains this vertex.
We repeat this procedure until there is no such vertex left. Let $G'$ denote
the remaining graph. Clearly by our procedure at most $(n-1)(d/r)<nd/r=e$ edges have been
removed in the process. So $G'$ is nonempty. Also, by our condition, $\delta(G')\geq d/r$.
\end{proof}

Below  we give a version of the Chernoff bound from \cite{MR}. 
\begin{lemma} {\bf (Chernoff bound)} \label{chernoff}
Let $X$ be the sum of $n$ independent random variables $X_1,\ldots, X_n$, where
for each $i\in [n]$, $\pr(X_i=1)=p$ and $\pr(X_i=0)=1-p$. Then for any real $0\leq \alpha\leq 1$
$$\pr(|X-np|>\alpha np)<2 e^{-\frac{\alpha^2}{3} np}.$$
\end{lemma}

Recall that given a hypergraph $G$ and a vertex $x$, the {\it link graph} $L_G(x)$ of $x$ in $G$
is the graph $\{e\setminus \{x\}, e\in G, x\in e \}$. Given set $S$ of vertices in $G$, the
{\it subgraph $G[S]$ of $G$ induced by $S$} is the graph with vertex set $S$ and edge
set $\{e: e\in G, e\subseteq S\}$.

\begin{proposition} \label{split}
Let $c>0$ be a fixed real. Let $m,r,t\geq 2$ be fixed positive integers. There exists a positive integer $n_0$ depending on $c,m,r,t$ such that for all $n\geq n_0$ the following holds.
Let $G$ be a linear $r$-graph with $\delta(G)\geq  cn^{\frac{1}{m}}$. Then there exists
a partition of $V(G)$ into $t$ sets $S_1,\ldots, S_t$ such that 
for each $u\in V(G)$ and each $i\in [t]$, $|L_G(u)\cap G[S_i]|\geq \frac{c}{2t^{r-1}}n^{\frac{1}{m}}$.
\end{proposition} 
\begin{proof}
Independently and uniformly at random assign each vertex in $G$ a color from $[t]$.
For each $i\in [t]$ let $S_i$ be the set of vertices receiving color $i$.
For each $u\in V(G), i\in [t]$, let $Y_{u,i}$ be the random variable that counts the
number of edges in $L_G(u)$ completely contained in $S_i$. For fixed $u,i$,
clearly each edge of $L_G(u)$ has probability $\frac{1}{t^{r-1}}$ of being contained in $S_i$.
Since $G$ is a linear $r$-graph, the edges of $L_G(u)$ are pairwise vertex-disjoint.
So $Y_{u,i}$ is the sum of $d(u)$ independent random variables each of which equals $1$ with
probability $p=\frac{1}{t^{r-1}}$ and $0$ with probability $1-p$. By Lemma \ref{chernoff},
$$\pr\left(Y_{u,i}<\frac{1}{2}\frac{d(u)}{t^{r-1}}\right)<
\pr\left(|Y_{u,i}-\frac{d(u)}{t^{r-1}}|>\frac{1}{2}\frac{d(u)}{t^{r-1}}\right)<2 \exp\left(-\frac{1}{12}\frac{d(u)}{t^{r-1}}\right).$$
Since $d(u)\geq cn^{\frac{1}{m}}$, this yields
$$\pr\left(Y_{u,i}<\frac{cn^{\frac{1}{m}}}{2t^{r-1}}\right)< 2 \exp\left(-\frac{cn^{\frac{1}{m}}}{12t^{r-1}}\right).$$
Thus, 
$$\pr\left(\exists u\in V(G), \exists i\in [t], Y_{u,i}<\frac{cn^{\frac{1}{m}}}{2t^{r-1}}\right)
<2tn \cdot \exp\left(-\frac{cn^{\frac{1}{m}}}{12 t^{r-1}}\right)<1,$$
for all $n\geq n_0$, where $n_0$ depends only on $c, m, r$, and $t$. 
Thus there exists a particular coloring for which $Y_{u,i}\geq
\frac{cn^{\frac{1}{m}}}{2m^{r-1}}$ for all $u\in V(G)$ and $i\in [t]$. Let $S_1,\ldots, S_t$ 
be the color classes of this coloring. Then $(S_1,\ldots, S_t)$ forms a desired partition.
\end{proof}

%%%%%%%%%%%%%%%%%%%%%%%%%%%%%%%%%%%%%%%%%%%%%%%%%%%%%%%%%%%%%

\section{Linear Tur\'an numbers of $r$-uniform even cycles}

The following lemma provides the  main ingredient of our proof of Theorem \ref{main} for even cycles. 

\begin{lemma} \label{even-expansion}
Let $r,m,h$ be fixed integers, where $r\geq 3, m\geq 2, 0\leq h\leq m-1$.
Let positive integer $i$, let $c_i=\frac{1}{(rm2^{r+2})^i}$.
Let $G$ be a linear $r$-graph such that $C^{r}_{2m}\not\subseteq G$.
Let $H$ be an $r$-uniform leveled linear trees of height $h$ rooted at $w$ that is contained
in $G$. Let $L_0,\ldots, L_h$ denote the levels of $H$. Let $E$ be a set of edges in $G$
each of which contains one vertex in $L_h$ and $r-1$ vertices outside $H$. Suppose
that $|E|\geq (m2^{r+3})^h |L_h|$. Then there exists
a subset $E^*$ of $E$ such that $|E^*|\geq c_h |E|$ and that $E^*\setminus L_h$ is a matching.
In particular, $H\cup E^*$ is a leveled
linear trees of height $h+1$ rooted at $w$, with $L_{h+1}$ consist of one vertex of $e\setminus 
L_h$ for each $e\in E^*$.
\end{lemma}
\begin{proof}
We use induction on $h$. For the basis step let $h=0$ and $H$ consists of a single vertex $w$. By our assumption, $E$ is a set of edges containing $w$.
Since $G$ is linear, every two of these edges intersect only at $w$. 
Let $E^*=E$. It is easy to see that the claim holds.

For induction step, let $h\geq 1$. Suppose $T$ is a $2$-uniform
tree of height $h$ rooted at $w$ with levels $L_0,L_1,\ldots, L_h$ and $H=T^{(r)}\subseteq G$.
By our assumption, each edge in $E$ contains one vertex in $L_h$ and $r-1$ vertices outside $H$.  Let $F=\{e\setminus L_h: e\in E\}$. Then $F$ is an $(r-1)$-graph.
Since $G$ is linear and $r\geq 3$, 
the mapping $\sigma: E\to F$ that maps $e$ to $e\setminus L_1$ is a bijection.
So $|F|=|E|$. Let $Q$ be a minimum vertex cover of $F$.
By Lemma \ref{cross-cut}, there exist $F'\subseteq F$ and $Q'\subseteq Q$ such
that $|F'|\geq \frac{r-1}{2^{r-1}}|F|=\frac{r-1}{2^{r-1}}|E|$ and that $Q'$ is a cross-cut of $F'$.
Let $E'$ be the set of edges of $E$ corresponding to edges of $F'$ (via $\sigma^{-1}$).
Then $|E'|=|F'|$ and each edge of $E'$ contains exactly one vertex of $L_h$,
one vertex of $Q'$, and $r-2$ vertices outside $V(H)\cup Q'$. 
Let $B=\{e\cap (L_h\cup Q'): e\in E'\}$. 
By definition, $B$ is a bipartite $2$-graph with a bipartition $(X,Q')$ where $X=V(B)\cap L_h$.
The mapping $f: e\to e\cap(L_h\cup Q')$ is a bijection from $E'$ to $B\subseteq \partial_2(G)$.
So
$$|B|=|E'|=|F'|\geq \frac{r-1}{2^{r-1}}|E|.$$
Clearly, no edge of $G$ contains more than one edge of $B$ and
in the default edge-coloring $\phi$ of $B$ the colors are disjoint from $V(B)\cup V(H)$.

Let $x_1,\ldots, x_p$ denote the children of $w$ in $T$. For each $i\in [p]$, let 
$A_i=V(T_{x_i})\cap L_h$. So $A_i$ consists of vertices in $L_h$ that are descendants
of $x_i$ (in $T$). Note that $A_1,\ldots, A_p$ are pairwise disjoint.
Let
\begin{align*}
Q^+&=\{x\in Q': N_B(x) \cap A_i \neq \emptyset \mbox{ for at least
$2rm$ different $A_i$'s}\}\\
Q^-&=\{x\in Q': N_B(x) \cap A_i \neq \emptyset \mbox{ for fewer than
$2rm$ different $A_i$'s}\}\\
\end{align*}
Then $Q^+$ and $Q^-$ partition $Q'$.  Let $B^+$ denote the subgraph of $B$ induced by $X\cup Q^+$ and $B^-$ the subgraph of $B$ induced by $X\cup Q^-$. Then $B=B^+\cup B^-$.

\medskip

{\bf Claim 1.} $|Q'|\geq \frac{c_{h-1}|B|}{8rm}$.

\medskip

{\it Proof of Claim 1.} We consider two cases.

\medskip

{\bf Case  1.} $|B^+|\geq \frac{1}{2}|B|$.

\medskip

By our earlier discussion, $|B|\geq \frac{r-1}{2^{r-1}}|E|$. So
$|B^+|\geq \frac{|B|}{2}>\frac{(r-1)|E|}{2^r}$. We claim that $|Q^+|\geq \frac{|B^+|}{4rm}$.
Suppose for contradiction that $|Q^+|< \frac{|B^+|}{4rm}$. Then $|B^+|\geq 4rm |Q^+|$.
By our assumption $|E|\geq (m2^{r+3})^h|L_h|\geq (m2^{r+3})^h |X|$. 
Hence $|B^+|\geq \frac{(r-1)(m2^{r+3})^h}{2^r}|X|\geq 4rm|X|$.
So $|B^+|\geq 2rm (|X|+|Q'|)=2rm |V(B^+)|$.
Thus, $B^+$ has average degree at  least
$4rm$. By a well-known fact, $B^+$ contains a subgraph $B^*$
with minimum degree at least
$2rm$. Let $\phi$ be the default edge-coloring of $B^*$.
By Lemma \ref{default-coloring},  $\phi$ is strongly proper. Let $x$ be
any vertex in $V(B^*)\cap Q^+$. By Lemma \ref{rainbow-path}, $B^*$ contains a path
$P$ of length $2m-2h-2$ starting at $x$ that is strongly rainbow under $\phi$.
Since $B^*$ is bipartite  and $2m-2h-2$ is even, the other endpoint $y$ of $P$ lies in $Q^+$.
Now the $r$-graph $P^+$ with edge set $\{e\cup \phi(e): e\in P\}$ is a linear path of length
$2m-2h-2$ with endpoints $x$ and $y$ using edges of $E'\subseteq E$. 
By the definition of $E$, $V(P^+)\cap V(H)\subseteq L_h$.

Now, since $x\in Q^+$, by definition,  $N_B(x)\cap A_i\neq \emptyset$ for at least
$2rm$ different $i$'s. Without loss of generality suppose $N_B(x)\cap A_i\neq \emptyset$
for $i=1,\ldots, 2rm$. For each $i\in [2rm]$, let $u_i\in N_B(x)\cap A_i$ 
and let $e_i$ be the unique
edge of $E$ containing $xu_i$.  Since $G$ is linear,
$e_1\setminus\{x\},\ldots, e_{2rm}\setminus \{x\}$ are pairwise disjoint. Since there are clearly
fewer than $2rm$ vertices contained in $P^+$, for some $i\in [2rm]$, $e_i\setminus \{x\}$
is vertex disjoint from $P^+$. Without loss of generality, suppose $e_1\setminus\{x\}$ is
vertex disjoint from $P^+$.
 Likewise, since $y\in Q^+$, we can find an edge $f_1\in E'$ containing $y$ intersecting some $A_j$ such that $j\neq 1$ and that $f_1\setminus \{y\}$ is disjoint from  $V(P^+)\cup e_1$. 
Without loss of generality, suppose $j=2$.
Let $\{v_1\}=f_1\cap A_2$. Let $P_1$ be the unique $u_1,w$-path
and $P_2$ the unique $v_1,w$-path in $H$, respectively. Since $x_1$ and $x_2$
are different children of $w$ in $T$, $P_1,P_2$ are two internally disjoint paths of length $h$,
sharing only $w$. Now  $P^+\cup\{e_1,f_1\}\cup P_1\cup P_2$ is a
linear cycle of length $2m-2h-2+2+2h=2m$ in $G$, contradicting our assumption about $G$.
Hence $|Q^+|\geq \frac{|B^+|}{4rm}\geq \frac{|B|}{8rm}$ and thus $|Q'|\geq \frac{|B|}{8rm}\geq 
\frac{c_{h-1}|B|}{8rm}$.

\medskip

{\bf Case 2.} $|B^-|\geq \frac{1}{2}|B|$.

\medskip

We have $|B^-|\geq \frac{|B|}{2} \geq \frac{(r-1)|E|}{2^r}$. 
For each vertex $x\in B^-$, by our assumption, $N_B(x)\cap A_i\neq \emptyset$
for fewer than $2rm$ different $i$'s. Among the $A_i$'s that receive edges of $B^-$ from $x$,
let $A_{i(x)}$ be one that receives the most edges of $B$ from $x$. 
We now form a subgraph $B^-_1$ of $B^-$ by including for each $x\in Q^-$
the edges from $x$ to $A_{i(x)}$.  By our procedure, 
\begin{equation} \label{B1}
|B^-_1|\geq \frac{|B^-|}{2rm}\geq \frac{(r-1)|E|}{rm 2^{r+1}}\geq \frac{r-1}{rm2^{r+1}} (m2^{r+3})^h|L_h|\geq 2(m2^{r+3})^{h-1}|L_h|.
\end{equation}
Recall that $A_1,\ldots, A_p$ are disjoint subsets of $L_h$. In $B^-_1$, each vertex in $Q^-$ sends edges
to at most one $A_i$. For each $A_i$, call $A_i$ {\it light} if the number of edges of $B^-_1$
incident to $A_i$ is less than $(m2^{r+3})^{h-1}|A_i|$; otherwise call $A_i$ {\it heavy}. 
Clearly the total number of edges of $B^-$ that are incident to light $A_i$'s is at most
$(m2^{r+3})^{h-1} |L_h|$, which is at most $\frac{1}{2}|B^{-}_1|$ by \eqref{B1}.
So the number of edges of $B^-_1$ that are incident to heavy
$A_i$'s is at least $\frac{1}{2}|B^-_1|$. 

Without loss of generality, suppose that $A_1,\ldots, A_t$ are the heavy $A_i$'s.
For each $i\in [t]$, let $Q_i$ be the set of vertices in $Q^-$ that 
are joined by edges of $B^-_1$ to $A_i$. 
By our definition of $B^-_1$, $Q_1,\ldots, Q_t$ are pairwise disjoint.
Also, for each $i\in [t]$, let $E_i$ be the set of edges of $E'$ corresponding to the
set of edges of $B^-_1$ that are  incident to $A_i$.  By our assumption $|E_i|\geq  (m2^{r+3})^{h-1}|A_i|$. Recall that $x_1,\ldots, x_p$ denote the children of $w$ in $T$.
For each $i\in [t]$, $H_{x_i}$ is
a linear tree of height $h-1$ rooted at $x_i$ whose $(h-1)$-th level is $A_i$.
Each edge of $E_i$ contains one vertex of $A_i$ and $r-1$ vertices outside $H_{x_i}$
and $|E_i|\geq (m2^{r+3})^{h-1}|A_i|$. By induction hypothesis, there exists $E'_i\subseteq 
E_i$ such that $|E'_i|\geq c_{h-1}|E_i|$ and $E'_i\setminus A_i$ is a matching. In particular,
this yields $|Q_i|\geq c_{h-1}|E_i|$. Hence $|Q'|\geq |Q^-|\geq \sum_{i=1}^t |Q_i|
\geq c_{h-1}\sum_{i=1}^t |E_i|\geq c_{h-1}\frac{|B^-_1|}{2}\geq \frac{c_{h-1}|B^-|}{4rm}
\geq \frac{c_{h-1}|B|}{8rm}$ by \eqref{B1}.
This proves Claim 1. \qed

\medskip

By Claim 1, we have $|Q|\geq |Q'|\geq \frac{c_{h-1}|B|}{8rm}
\geq \frac{(r-1)c_{h-1}|E|}{2^{r-1}8rm} =\frac{(r-1)c_{h-1}|E|}{rm2^{r+2}}$.
By Lemma \ref{matching-cover}, $F$ contains a matching $F^*$  of size at least
$\frac{c_{h-1}|E|}{rm2^{r+2}}=c_h |E|$. 
Let $E^*$ be the set of edges of
$E$ corresponding to $F^*$. Then $|E^*|=|F^*|$ and 
$H\cup E^*$ is a leveled linear tree of height $h+1$ rooted at $w$ with $L_{h+1}$ consisting
of one vertex of each edges in $F^*$.
\end{proof}

\begin{theorem} \label{even-cycles}
Let $m,r$ be positive integers where $m\geq 2$ and $r\geq 3$.
There exist a positive real $c_{m,r}$ and a positive integer $n_1$
such that for all $n\geq n_1$ we have
$$ex_L(n, C^{r}_{2m})\leq c_{m,r} n^{1+\frac{1}{m}}.$$
\end{theorem}
\begin{proof}
Let $\beta=(rm2^{r+2})^m$ and 
$c_{m,r}=2m^{r-1}\beta$. Choose $n_1$ such that $c_{m,r}n_1^\frac{1}{m}\geq n_0$, where $n_0$ is given in Lemma \ref{split}. Let $G$ be an $n$-vertex linear $r$-graph with
at least $c_{m,r}n^{1+\frac{1}{m}}$ edges, where $n\geq n_1$. We prove that $G$ contains a copy of $C^r_{2m}$.
By our assumption, $G$ has average degree at least $rc_{m,r} n^{\frac{1}{m}}$.
By Lemma \ref{min-degree}, there exists a subgraph $G'$ of $G$ with $\delta(G')\geq 
c_{m,r}n^{\frac{1}{m}}$. Let $N=n(G')$. Then $N\geq c_{m,r} n^\frac{1}{m}\geq n_0$ and $\delta(G')\geq c_{m,r} N^{\frac{1}{m}}$. By Lemma \ref{split} (with $t=m$),
there exists a partition of $V(G')$ into $S_1,\ldots, S_m$
such that for each $u\in V(G')$ and $i\in [m]$, we have $|\cL_{G'}(u)\cap G'[S_i]|\geq
\frac{c_{m,r}}{2m^{r-1}} N^{\frac{1}{m}}=\beta N^{\frac{1}{m}}$.

Let $w$ be any vertex in $S_1$. Let $L_0=\{w\}$. Inside $G'$,
we will construct a leveled linear tree $H$ of height $m$ 
rooted at $w$ with levels $L_1,\ldots, L_m$ such that for each $i\in [m]$,
$L_i\subseteq S_i$ and $|L_i|\geq N^{\frac{1}{m}}|L_{i-1}|$. This will imply that
$|L_m|\geq N$, which is a contradiction, which will then complete our proof.

We construct $H$ as follows.
Let $E_1$ be the set of edges of $G'$ containing $w$ that correspond to $\cL_{G'}(w)\cap G'[S_1]$.
By our assumption, $|E_1|\geq \beta N^{\frac{1}{m}}\geq N^{\frac{1}{m}}$, by our definition of
$\beta$. Also, each edge of $E_1$
consists of $w$ and $r-1$ vertices in $S_1$. Let $L_1$ consists of a vertex from
$e\setminus \{w\}$ for each $e\in E_1$. 
In general, suppose we have grown $i$ levels $L_1,\ldots, L_i$, where $i\leq m-1$, 
such that for each $j\in [i]$, $L_j\subseteq S_j$ and $|L_j|/|L_{j-1}|\geq N^\frac{1}{m}$.
Let $E_i$ denote the set of edges in $G'$ that contain one vertex in $L_i$
and $r-1$ vertices in $S_{i+1}$. By our assumption about the partition $(S_1,\ldots, S_m)$, 
$|E_i|\geq \beta N^{\frac{1}{m}} |L_i|\geq (m2^{r+3})^m|L_i|$, noting that $\beta\geq
(m2^{r+3})^m$.
Since $C^r_{2m}\not\subseteq G'$, by Lemma \ref{even-expansion},  there exists
a subset $E^*_i\subseteq E_i$ such that $|E^*_i|\geq \frac{1}{(rm2^{r+2})^i} |E_i|$ such
$E^*_i\setminus L_i$ is a matching.  Let $H_{i+1}=H_i\cup E^*_i$
and let $L_{i+1}$ consists of one vertex from $e\setminus L_i$ for each $e\in E^*_i$.
Then $H_{i+1}$ is a leveled linear tree rooted at $w$ of height $i+1$ whose $i+1$-th level
$L_{i+1}$ is contained in $S_{i+1}$. Furthermore, $|L_{i+1}|=|E^*_i|\geq 
\frac{1}{(rm2^{r+2})^i} |E_i|\geq \frac{\beta}{(rm2^{r+2})^i}N^{\frac{1}{m}}|L_i|\geq N^{\frac{1}{m}}|L_i|$. We can continue like this to construct $H$ and derive the desired contradiction.
\end{proof}

%%%%%%%%%%%%%%%%%%%%%%%%%%%%%%%%%%%%%%%%%%%%%%%%%%%%%%%%%%%%%%%
\section{Leveled linear quasi-trees}

\subsection{Leveled linear quasi-trees}

To study the odd cycle case, we generalize the notion of leveled linear trees as follows. Let $r\geq 3$. A linear $r$-graph $H$ is called a {\it leveled linear quasi-tree} of height $h$ {\it rooted at} $w$ if it is the union of a sequence of $r$-graphs $H_0, H_1,\ldots, H_{h-1}$ satisfying the following: (1) Each $H_i$ is an $r$-partite $r$-graph with no isolated vertex and has parts $L_i, L'_i, J^{(1)}_i,\ldots, J^{(r-2)}_i$ such that with $B_i=\{e\cap (L_i\cup L'_i): e\in H_i\}$, $H_i$ is the $r$-expansion of $B_i$.
% with $J^{(1)}_i\cup \ldots \cup J^{(r-2)}_i$ forming the set of expansion vertices. 
(2) For each $i=0,1,\ldots, h-1$, $J^{(r-2)}_i=L_{i+1}$. (3) For each $i=0,1,\ldots, h-2$, $V(H_i)\cap V(H_{i+1})=L_{i+1}$ and $V(H_i)\cap V(H_j)=\emptyset$ whenever $|i-j|>1$. 
(4) $L_0=\{w\}$. For each $i=0,\ldots, h$, we call $L_i$ the $i$th {\it main level} of $H$. 
For each $i=0,\ldots, h-1$, we call $L'_i$ the $i$th {\it companion level} of $H$.

For each $i\in \{0,1,\ldots, h-1\}$, we call $H_i$ the $i$-th {\it segment} of $H$ and
$B_i$ {\it the defining bipartite graph} of $H_i$.  For each edge $f$ of $B_i$
the unique vertex in $L_{i+1}$ that corresponds to $f$ is said to be a {\it presentative}
of $e$. Given $x\in V(B_i)$ and $y\in L_{i+1}$,
we say that $y$ is a {\it child} of $x$ and that $x$ is a {\it parent} 
of $y$ if $y$ is a representative of an edge of $B_i$ incident to $x$. Observe
that every two different vertices $u,v$ in the same main level $L_i$ or in
the same companion level $L'_i$, where $i\leq h-1$, 
must have disjoint sets of children in $L_{i+1}$ since the sets of edges of $B_i$ incident
to $u$ and $v$, respectively, are disjoint. 

Given a vertex $x\in L_i\cup L'_i$, where $i\leq h-1$, define  the {\it down tree } $T_x$, rooted at $x$, to be the $2$-graph obtained by
including all the edges between $A_0=\{x\}$ and its set $A_1$ of children 
in $L_{i+1}$, and then including all the edges joining vertices in $A_1$ and 
the set $A_2$ of their children in $L_{i+2}$ and etc, until we run out of levels. It is easy to see that $T_x$ is a tree rooted at $x$ of height at most $h-i$. 
Also, if $x,y\in L_i$ or $x,y\in L'_i$, $x\neq y$, then the earlier observation about
disjoint sets of children implies that $V(T_x)\cap V(T_y)=\emptyset$. Furthermore, in $T_w$,
where $w$ is the root of $H$, for each $i=0,\ldots, h$, 
the $i$-th distance class from $w$ is precisely all of $L_i$.

Given a vertex $x\in L_i\cup L'_i$, where $i\leq h-1$, define the {\it down graph} $H_x$, rooted at $x$, to be the subgraph of $H$ obtained by replacing each edge $f$ of $T_x$ with the corresponding edge $e$ of $H$ that contains $f$. The following lemma follows immediately from the definitions and our discussions above.

\begin{lemma} \label{down-graph}
Let $H$ be an $r$-uniform leveled linear quasi-tree of height $h$ rooted at $w$ with 
segments $H_1,\ldots, H_{h-1}$. Let $x\in L_i\cup L'_i$, where $0\leq i\leq h-1$. 
Then $H_x$ is a leveled quasi-tree of height at most $h-i$ rooted at $x$. Also, 
$\forall a,b\in L_i\cup L'_i$, $a\neq b$, if either 
$a,b\in L_i$ or $a,b\in L'_i$,  then $(V(H_a)\cap V(H_b))\cap L_j=\emptyset$ for all $j\geq i+1$.
\end{lemma}

In a linear $r$-graph, a path $P$ is just the $r$-expansion of a $2$-uniform path.
An endpoint of $P$ is a vertex in the first or last edge that has degree $1$ in $P$.
An $x,y$-path is a path where $x$ is an endpoint in the first edge of $P$ and $y$ is
an endpoint in the last edge of $P$ (or vice versa).

\begin{lemma} \label{joining-path}
Let $H$ be an $r$-uniform leveled linear quasi-tree of height $h$ rooted at $w$ with 
segments $H_1,\ldots, H_{h-1}$, where $L_0,L_1,\ldots, L_h$ and $L'_0,\ldots, L'_{h-1}$ denote the main levels and companion levels, respectively.
Let $x,y\in L_i, x\neq y$, where $1\leq i\leq h-1$. Then there exists an $x,y$-path $P$ of an
even length at most $2i$ that is contained in $\bigcup_{j=0}^{i-1} H_j$ and intersects
$L_i$ only in $x$ and $y$.
\end{lemma}
\begin{proof}
We use induction on $i$. The claim is trivial when $i=1$. So assume $i\geq 2$.
Let $e$ be the unique edge of $H_i$ that contains $x$ and $f$ the unique edge of $H_i$
that contains $y$. If $e$ and $f$ share a vertex, then $e\cup f$ is an $x,y$-path of
length $2$.  Otherwise $e\cap f=\emptyset$. Let $\{x'\}=e\cap L_{i-1}$ and $\{y'\}=f\cap 
L_{i-1}$. By induction hypothesis, there is an $x',y'$-path $P$ of an even length at most $2(i-1)$ 
that is contained in $\bigcup_{j=0}^{i-2} H_j$ and intersects $L_{i-1}$ only in $x'$ and $y'$.
Now, $P\cup \{e,f\}$ is an $x,y$-path of an even length at most $2i$ that is contained
in $\bigcup_{j=0}^{i-1} H_j$ and intersects $L_i$ only in $x$ and $y$.
\end{proof}

Given a leveled linear quasi-tree $H$ rooted at $w$, a {\it monotone path} is a path  in $H$
that hits each main level at most once.  It is easy to see that for every vertex $x$ in the $i$-th main level, there is a unique monotone $w,x$-path, and that path has length $i$. For every vertex $y$ in the $i$th companion level, there exists at least one monotone $w,y$-path and 
such a path has length $i+1$.

An $r$-uniform {\it spider} $F$ with $t$ legs consists of $p$ $r$-uniform linear paths $P_1,\ldots, P_t$ (called the {\it  legs}) sharing one endpoint $x$ but are otherwise vertex-disjoint.

\begin{lemma} \label{spider}
Let $h,p,r$ be positive integes, where $r\geq 3$.
Let $H$ be an $r$-uniform leveled linear quasi-tree of height $h$ rooted at $w$ with 
segments $H_1,\ldots, H_{h-1}$. Let $L_0,\ldots, L_h$ and $L'_0, \ldots, L'_{h-1}$ 
be the main levels and companion levels, respectively.
Let $S\subseteq L_h$ such that $|S|\geq (hpr)^h$. Then exists a vertex $x\in V(H)$ such that
 1) $|V(H_x)\cap S|\geq \frac{1}{(hpr)^{h-1}}|S|$ and (2) $H_x$ contains a spider centered at $x$  that has $p$ legs each of which is a monotone path from $x$ to $V(H_x)\cap S$.
\end{lemma}
\begin{proof}
We use induction on $h$. For the basis step let $h=1$. In this case, the claim clearly holds
by choosing $x$ to be $w$ and $p$ of the edges containing $x$ to form the required spider.
For the induction step, let $h\geq 2$. Clearly there is at least one monotone path from the root
$w$ to $S$, so there exist spiders centered at $w$ with legs being monotone paths from
$w$ to $S$. Let us call these {\it $(w,S)$-spiders}.
 Among all $(w,S)$-spiders, let $M$ be one that has the maximum number of legs.
If $M$ has $p$ legs, then the claim holds with $x=w$. So assume $M$ has fewer than $p$ legs.
For each $y\in S$, let $P_y$ be the unique monotone path in $H$ from $w$ to $y$. 
The maximality of $M$ implies that each $y\in S$,  $P_y$ intersects $M$ somewhere
besides at $w$. Let $y\in S$. If $P_y$ intersects $M$ at a vertex $u$ in 
$V(H_i)\setminus\{L_i, L'_i\}$ for some $i\leq h-1$ then
such a vertex is an expansion vertex  in $H_i$  and  both $P_y$ and $M$ must contain the corresponding  edge $e$ of $H_i$ that contains $u$ and hence both contain $e\cap L_i$
and $e\cap L_{i+1}$.
Thus, for each $y\in S$, $P_y$ contains
a vertex in  $U=(V(F)\setminus \{w\})\cap (\bigcup_{i=1}^{h-1} (L_i\cup L'_i))$.

Since $U$ has fewer than $phr$ vertices, by the pigeonhole principle,
there exists a vertex $z$ in $U$ that is contained in at least $\ce{\frac{s}{hpr}}$ different $P_y$'s.
Suppose that $z\in L_a\cup L'_a$.
Let $S'$ be the set of vertices $y$ in $S$ such that $P_y$ contains $z$. Then $|S'|\geq
\frac{|S|}{hpr}$. For each $y\in S'$, let $P'_y$ be the $z,y$-path contained in $P_y$. 
Let $H'=\bigcup_{y\in S} P'_y$. Then $H'\subseteq H_z$.
Now, $H_z$ is a leveled linear quasi-tree with height at most $h-1$ and
$S'$ is a set of vertices in its last level. 
By the induction hypothesis, there is a vertex $x$ in $H_z$ such that $|V((H_z)_x)\cap S'|\geq
\frac{|S'|}{[(h-1)pr]^{h-2}}\geq \frac{|S|}{(hpr)^{h-1}}$ and that $(H_z)_{x}$ contains a $(z, V(H_z(x))\cap S')$-spider with $p$ legs. Consider now the relationship between $(H_z)_{x}$ and
$H_x$.  Since $x$ sends multiple internally disjoint monotone paths to $S$ it is easy to see
that either $x=z$ or $x\in L_j\cup L'_j$ for some $j\geq a+1$. In either case, we have $(H_z)_{x}=H_x$.
\end{proof}

%%%%%%%%%%%%%%%%%%%%%%%%%%%%%%%%%%%%%%%%%%%%%%%%%%%%%%%%%

\section{Linear Tur\'an numbers of odd cycles}

The following lemma provides the key ingredient for our proof of Theorem \ref{main}
for odd cycles. Before presenting the technical details, let us point out what the main
technical challenge is for the odd cycle case and what the key new ideas are in overcoming
the difficulty. The general plan is similar to the even cycle case. We use a linear quasi-tree
as a framework for growing levels and argue that in the absence of $C^r_{2m+1}$ the graph
must expand quickly. The main diffculty we face is that linear quasi-trees have a interweaving structure and no longer possess a clean tree structure. Therefore, we cannot hope to link vertices cleanly back to the root. The key idea to overcome this difficulty is to apply Lemma \ref{spider} to  locate a set of vertices (called ``dominators'') at some earlier level to act as a group of roots for different vertices. This idea of untangling via a buffer can be useful elsewhere.

\begin{lemma} \label{odd-expansion}
Let $r,m, h$ be integers, where $r\geq 3$, $m\geq 2$, and $0\leq h\leq m-1$.
Let $p=2mr$ and $c=2^{r+2} (mpr)^m$.
Let $G$ be a linear $r$-graph such that $C_{2m+1}\not\subseteq G$. 
Let $H$ be an $r$-uniform leveled linear
quasi-tree of height $h$ in $G$ rooted at $w$  with
segments $H_0, H_1,\ldots, H_{h-1}$, levels $L_0,L_1,\ldots, L_h$ and companion levels $L'_1,\ldots, L'_{h-1}$. Let $S$ be a set of vertices in $G$ outside $H$ and
$E$ a set of edges in $G$ each of which contains one vertex in $L_h$ and $r-1$ vertices outside
$H$. Suppose that $|E|\geq c^h|L_h|$. Then there exists a subset $E^*$ of $E$ and a set $S$ of vertices outside $H$ such that (1) $|E^*|\geq \frac{1}{c^h}|E|$, (2) $S$ is a cross-cut of $E^*$,
(3) $E^*$ is the $r$-expansion of the $2$-graph $\Gamma=\{e\cap (L_h\cup S): e\in E^*\}$
and 
(4) either $\delta(\Gamma)\geq p$ or each vertex in $S$ has degree $1$ in $\Gamma$.
In particular, $H\cup E^*$ is a leveled linear quasi-tree of height $h+1$ rooted at $w$, where $L'_h=S$ and $L_{h+1}$ consists of one vertex 
from each member of $E^*\setminus (L_h\cup S)$.
\end{lemma}

\begin{proof}
We use induction on $h$. For the basis step, let $h=0$. Then $H$ consists of the single vertex $w$ and $E$ is a set of edges containing $w$. Let $E^*=E$ and let $S$ consist of one vertex
of $e\setminus \{w\}$ for each $e\in E^*$. It is easy to see that the claim holds.

For the induction step, let $h\geq 1$. Let $E$ be defined as in the statement of the lemma.
Let $F=\{e\setminus L_h: e\in E\}$. Then $F$ is an $(r-1)$-graph with $|F|=|E|$. Let
$Q$ be a minimum vertex cover of $F$. 
First suppose that $|Q|\geq \frac{r-1}{c^h}|E|$. 
By Lemma \ref{matching-cover}, $F$ contains a matching $F^*$ of
size at leat $\frac{|Q|}{r-1}\geq  \frac{1}{c^h}|E|$. Let $E^*$ be the set of edges of $E$ corresponding to $F^*$. Let $S=E^*\cap Q'$. It is easy to check that $E^*$ and $S$ satisfy the four conditions and we are done. We henceforth assume that
\begin{equation} \label{Q-small}
|Q|<\frac{r-1}{c^h}|E|.
\end{equation}

By Lemma \ref{cross-cut}, there exists $F'\subseteq F$ and $Q'\subseteq Q$ such that $|F'|\geq \frac{r-1}{2^{r-1}}|F|$ and that $Q'$ is a cross-cut of $F'$. Let $E'$ be the set of edges in $E$ corresponding to $F'$. Then $|E'|=|F'|$ and each
edge in $E$ intersects each of $L_h$ and $Q'$ in exactly one vertex. Let $B=\{e\cap (L_h\cup Q'): e\in E'\}$. Then $B$ satisfies the condition of Lemma \ref{default-coloring} and 
there is a bijection between edges of $B$ and edges of $E'$. In particular, 
\begin{equation}\label{B-E}
|B|=|E'|\geq \frac{r-1}{2^{r-1}}|E|.
\end{equation}
Now, $|B|\geq \frac{r-1}{2^{r-1}}|E|\geq \frac{(r-1)c^h}{2^{r-1}}|L_h|\geq 32hmr|L_h|$.
Also, by \eqref{B-E}, $|B|\geq \frac{r-1}{2^{r-1}}|E|\geq \frac{c^h}{2^{r-1}}|Q'|\geq 32hmr|Q'|$.
So, $|B|\geq 32hmr(|L_h|+|Q'|)\geq 16hmr |V(B)|$. Thus $B$ has average degree at least $32hmr$.

We now partition $Q'$ as follows.
Let $$Q^-=\{y\in Q': d_B(y)< (mpr)^m\}  \mbox{ and  } Q^+=\{y\in Q': d_B(y)\geq (mpr)^m\}.$$
Let $B^-$ be the subgraph of $B$ induced by $L_h\cup Q^-$ and
the subgraph of $B$ induced by $L_h\cup Q^+$.

\medskip

{\bf Case 1.} $|B^-|\geq \frac{|B|}{2}$.

\medskip

In this case, we have $|Q|\geq |Q^-|>\frac{|B^-|}{(mpr)^m}\geq \frac{|B|} {2(mpr)^m}\geq
\frac{(r-1)}{2^r(mpr)^m}|E|>\frac{r-1}{c^h}|E|$,
 contradicting \eqref{Q-small}.

\bigskip

{\bf Case 2.}  $|B^+|\geq \frac{|B|}{2}$.

\medskip

In this case, we partition $Q^+$ as follows. Let $y\in Q^+$. Then $N_B(y)\subseteq
L_h$ and $|N_B(y)|\geq (mpr)^m$.  
By Lemma \ref{spider}, there exists $x\in V(H)$ such that
$|V(H_x)\cap N_B(y)|\geq \frac{|N_B(y)|}{(hpr)^{h-1}}$ and such that $H_x$ contains a spider
$M$ with $p$ legs from $x$ to $N_B(y)$ using monotone paths. 
Let us call such an $x$ a {\it dominator} of $y$ in $H$.
Suppose $V(M)\cap N_B(y)=\{y_1,\ldots, y_p\}$. For each $j=1,\ldots, p$, let $f_j$ be the unique edge of $E$ containing $yy_i$. It is easy to see that $M\cup \{f_1,\ldots, f_p\}$ form a collection of $p$ internally disjoint $x,y$-paths that intersect each main level of $H$ at most once. In particular, if $x\in L_i\cup L'_i$, where $i\leq h-1$, then these paths all have length $h-i+1$.

For each $y\in Q^+$, fix a dominator $\alpha(y)$ of $y$ in $H$.
Note that either $\alpha(y)=w$ or $\alpha(y)\in L_i\cup L'_i$ for some $1\leq i\leq h-1$, since
other vertices in $H$ have degree $1$ in $H$ and cannot possibly be a dominator of $y$.
Let $Q_0=\{y\in Q^+: \alpha(y)=w\}$. For each $i=1,\ldots, h-1$, let
$Q_i=\{y\in Q^+: \alpha(y)\in L_i\}$ and $Q'_i=\{y\in Q^+: \alpha(y)\in L'_i\}$.
Then $Q_0, Q_1,Q'_1,\ldots, Q_{h-1}, Q'_{h-1}$ partition $Q^+$.  Let $B_0$ denote the subgraph of $B^+$ induced by $Q_0\cup L_h$.  For each $i=1,\ldots, h-1$, let $B_i$ denote the subgraph of $B^+$ induced by $Q_i\cup L_h$ and $B'_i$ the subgraph of $B^+$ induced by $Q'_i\cup L_h$.
Then $B_0, B_1,B'_1,\ldots, B_{h-1}, B'_{h-1}$ partition $B^+$.
One of these graphs must then have size at least $\frac{|B^+|}{2h}$.

\medskip

{\bf Subcase 2.1.}  $|B_0|\geq \frac{|B^+|}{2h}$.

\medskip

In this subcase, we have $|B_0|\geq \frac{|B|}{4h}$.
Since $V(B_0)\subseteq V(B)$ and $B$ has average degree at least $32hmr$ by earlier discussion, $B_0$ has average at least $8mr$.
By Lemma \ref{subgraph}, $B_0$ contains a subgraph $B'_0$ with $\delta(B'_0)\geq 2mr$ 
and $|B'_0|\geq \frac{|B_0|}{2}\geq \frac{|B|}{8h}$.

\medskip

{\bf Claim 1.} The default coloring $\phi$ on $B'_0$ is strongly rainbow.

\medskip

{\it Proof of Claim 1.}
Let $e,e'\in B'_0$. First suppose that they are incident in $B'_0$.
Then since $\phi$ is strongly properly
on $B'_0$ by Lemma \ref{default-coloring}, we have $\phi(e)\neq \phi(e')$. Next,
suppose $e,e'$ are independent in $B'_0$.
Suppose for contradiction that $\phi(e)\cap \phi(e')\neq \emptyset$.
Let $v\in \phi(e)\cap \phi(e')$. Since $H$ is linear,
we have $\phi(e)\cap \phi(e')=\{v\}$.
Suppose $e=xy, e'=x'y'$ where $x,x'\in L_h$ and $y,y'\in Q_0$. Since $B'_0$ has minimum degree at least $2rm$, applying Lemma \ref{rainbow-path} with $k=r-2, \ell=2m-2-2h$ and $S_0=\phi(e)\cup \phi(e')$, $B'_0$ contains a path $P$ of length $2m-2-2h$ starting at $y'$
such that $P$ is strongly rainbow under $\phi$ and that $(\bigcup_{f\in P} \phi(f))\cap (\phi(e)\cup \phi(e'))=\emptyset$. Let $y''$ denote the other endpoint of $P$; it is possible that $y''=y'$.  The set of edges of $E$ that correspond to those in $P\cup \{e,e'\}$ forms
a linear path $R$ of length $2m-2h$ in which we may view $x$ as one endpoint at one end and $y''$ as an endpoint  at the other end. Let $R_x$ be a monotone path in $H$ from $w$ to $x$.
Then $R\cup R_x$ is a linear path of length $2m-2h+h=2m-h$ with $w$ being an endpoint
at one end and $y''$ being an endpoint at the other end. Since
$y''\in Q_0$, $w$ is a dominator of $y''$. Hence there exist $p$ pairwise internally disjoint
$w,y''$-paths of length $h+1$. Since $p=2mr>|V(R\cup R_x)|$, one of these paths, say $R'$,
is internally disjoint from $R\cup R_x$. Now $R\cup R_x\cup R'$ is a linear cycle of length 
$2m+1$ in $G$, a contradiction.  This proves Claim 1.\qed

\medskip

Let $E^*$ be the set of edges in $E$ corresponding to those in $B'_0$.
Then $|E^*|=|B'_0|\geq \frac{|B_0|}{2}\geq \frac{|B^+|}{4h}\geq \frac{|B|}{8h}\geq \frac{r-1}{2^{r+2}h}|E|$, by \eqref{B-E}. So, certainly  $|E^*|\geq \frac{1}{c^h}|E|$. Since $\phi$ is
strongly rainbow on $B_0$, $E^*$ is the $r$-expansion of $B'_0$. 
Also, $\delta(B'_0)\geq 2mr=p$. 
The lemma holds with $S=V(B'_0)\cap Q'$ and $\Gamma=B'_0$. 

\medskip

{\bf Subcase 2.2} $|B_i|\geq \frac{|B^+|}{2h}$ for some $1\leq i\leq h-1$.

\medskip

Fix such an $i$. We define a subgraph $D$ of $B_i$ as follows. For each $y\in Q_i$,
by definition, $\alpha(y)\in L_i$, we include all the edges of $B_i$
from $y$ to $V(H_{\alpha(y)})\cap N_B(y)$ in $D$. 
Since $\forall y\in Q_i$, $|V(H_{\alpha(y)}\cap 
N_B(y)|\geq \frac{|N_B(y)|}{(hpr)^{h-1}}$, using \eqref{B-E} we have
\begin{equation} \label{D-bound}
|D|\geq \frac{|B_i|}{(hpr)^{h-1}}\geq \frac{|B|}{4h(hpr)^{h-1}}\geq \frac{r-1}{2^{r+1}m(mpr)^{m-1}}|E|>\frac{c^h}{2^{r+1}m(mpr)^{m-1}}|L_h|.
\end{equation}

For each $x\in L_i$, let $A_x=L_h\cap V(H_x)$. By Lemma \ref{down-graph},
$\forall x,x'\in L_i, x\neq x'$ we have $A_x\cap A_{x'}=\emptyset$.
For each $x\in L_i$,
let $D_x$ denote the subgraph of $D$ consisting of edges of $D$ that are
incident to $A_x$ and let $C_x=V(D_x)\cap Q_i$. By our definition of $D$,
the sets $C_x$ are pairwise disjoint over different vertices $x$ in $L_i$.
Let $E_x$ denote the set of edges in $E$ correspond to the edges
of $D_x$.  Then $|E_x|=|D_x|$. Furthermore, each edge in $E_x$ contains
exactly one vertex in $A_x$ and exactly one vertex in $C_x$.

For each $x\in L_i$
we call $x$ {\it light} if $|D_x|\leq c^{h-1}|A_x|$ and
{\it heavy} if $|D_x|> c^{h-1}|A_x|$. Clearly, 
the combined size of $D_x$ over all light $x$ in $L_i$ is at most
$c^{h-1}|L_h|$. By our definition of $c$, one can check that 
$\frac{c^h}{2^{r+2}m(mpr)^{m-1}}\geq 2c^{h-1}$. So by \eqref{D-bound} and
our discussion above, 
\begin{equation}  \label{heavy}
|\bigcup\{D_x: x\in L_i, x \mbox{ is heavy }\}|\geq \frac{1}{2}|D|.
\end{equation}

Now, consider any  heavy $x\in L_i$. Since $H_x$ is a leveled linear quasi-tree rooted at
$x$ of height $h-i\leq h-1$ with last level $A_x$ and $E_x$ is a set of at least 
$c^{h-1}|A_x|$ edges each of which contains one vertex in $A_x$ and $r-1$ vertices
outside $H_x$, we can apply the induction hypothesis to obtain $E^*_x$ and $S_x$,
described as below.
Here $E^*_x$ is a subset of $E_x$ with $|E^*_x|\geq \frac{1}{c^{h-1}}|E_x|$,
$S_x^*$ is a cross-cut of $E^*_x$ outside $H_x$. 
By our definition of $E$, $S_x^*$ is outside $H$.
Further, the default edge-coloring of
$\Gamma_x=\{e\cap (L_h\cup S_x): e\in E^*_x\}$ is strongly rainbow and 
either $\delta(\Gamma_x)\geq p$ or $\forall v\in S_x, d_{\Gamma_x}(v)=1$.
We say that $x$ is of {\it type 1} if $\delta(\Gamma_x)\geq p$ and that 
$x$ is of {\it type 2} otherwise. 
Observe that if $x$ is of type 2, then  $E^*_x\setminus L_h$ is a matching of
size $|E^*_x|$. Since each edge in $E^*_x\subseteq E_x$ contains a vertex of $C_x$, 
this implies that $|C_x|\geq |E^*_x|$.

Let $L_{i,1}=\{x\in L_i, x \mbox{ is heavy and is of type 1} \}$
and $L_{i,2}=\{x\in L_i, x\mbox{ is heavy and is of type 2} \}$.
Suppose first that 
$\sum_{x\in L_{i,2}} |E_x|\geq \frac{1}{4}|D|$. Then 
$\sum_{x\in L_{i,2}}  |D_x| \geq \frac{1}{4}|D|$.
By \eqref{D-bound}, 
\begin{eqnarray*}
|Q|\geq |Q_i|\geq \sum_{x\in L_{i,2}} |C_x|\geq 
\sum_{x\in L_{i,2}} |E^*_x|\geq \frac{1}{c^{h-1}} \sum_{x\in L_{i,2}} |E_x|
\geq \frac{1}{4c^{h-1}}|D|\geq \frac{r-1}{c^{h-1}2^{r+3}m(mpr)^{m-1}}|E|\geq\frac{r-1}{c^h}|E|,
\end{eqnarray*}
contradicting \eqref{Q-small}. Hence, by \eqref{heavy}, we may assume that

\begin{equation}
\sum_{x\in L_{i,1}} |D_x| \geq \frac{1}{4}|D|.
\end{equation}

Recall that $\phi$ denotes the default edge-coloring of $\partial_2(G)$. Recall also that
$\forall x,y\in L_i$, $x\neq y$, we have  $A_x\cap A_y=\emptyset$ and $C_x\cap C_y=\emptyset$. So $V(\Gamma_x)\cap V(\Gamma_y)=\emptyset$.

\medskip

{\bf Claim 2.} Let $x,y\in L_{i,1}, x\neq y$.
Let $e\in \Gamma_x$ and $f\in \Gamma_y$. Then $\phi(e)\cap \phi(f)=\emptyset$.

\medskip

{\it Proof of Claim 2.} 
By Lemma \ref{joining-path}, there exists an $x,y$-path $R_0$ of some even length $2j\leq 2i$ in $\bigcup_{t\leq i} H_t$ that intersects $L_i$ only in $x$ and $y$. 
Since $x,y\in L_{i,1}$, we have $\delta(\Gamma_x)\geq p$ and $\delta(\Gamma_y)\geq p$. Suppose for contradiction that $\phi(e)\cap \phi(f)\neq 
\emptyset$. Let $v\in \phi(e)\cap \phi(f)$. Since $H$ is linear, we have
$\phi(e)\cap \phi(f)=\{v\}$. Suppose $e=ab$ and $f=a'b'$, where $a\in A_x, b\in C_x$
and $a'\in A_y, b'\in C_y$.  Let $\ell=2m-[2j+2(h-i)+2]=2m-2-2h+2(i-j)$. Note that $\ell$ is even and satisfies $2m-2-2h\leq \ell \leq 2m-4$.  
Since $\delta(\Gamma_y)\geq p=2mr>r\ell+2r$, by Lemma \ref{rainbow-path}, there
exists a path $P$ in $\Gamma_y$ of length $\ell$ starting at $b'$ that is strongly rainbow
under $\phi$ and such that $(\bigcup_{e'\in P} \phi(e'))\cap (\phi(e)\cup \phi(f))=\emptyset$.
Let $b''$ denote the other endpoint of $P$. Since $P$ has an even length, $b''\in C_y$.
Let $P^+$ denote the set of the edges of $E$ that correspond to the edges of $P\cup \{e,f\}$.
Then $P^+$ is linear path of length $2m-2h+2(i-j)$ in $G$ where $a$ is an endpoint at one
and $b''$ is an endpoint at the other end. Furthermore, $V(P^+)\cap V(H)\subseteq
L_h$. Let $R$ be a monotone path in $H$ from $x$ to $a$. Then $R$ has length $h-i$ and
is internally disjoint from $R_0$ and $P^+$. Since $b''\in C_y$, $y$ is a dominator of $b''$.
By definition, there exist $p$ internally disjoint $y,b''$-paths that intersect each main level
of $H$ at most once. In particular these paths have length $h-i+1$. Since $p=2mr>|V(P^+\cup R\cup R_0)|$, one of these paths, say $R'$, is internally disjoint from $P^+\cup R\cup R_0$.
Now $P^+\cup R\cup R_0\cup R'$ is  a linear cycle of length $2m-2h+2(i-j)+h-i+2j+h-i+1=2m+1$ in $G$, a contradiction. This proves Claim 2. \qed

\medskip

Now by Claim 2 and earlier discussion,  $\forall x,y\in L_{i,1}, x\neq y$,
we have $V(E_x)\cap V(E_y)=\emptyset$.
Let $E^*=\bigcup\{E_x: x\in L_{i,1}\}$,
$S=\bigcup\{C_x: x\in L_{i,1}\}$, and 
$\Gamma=\bigcup\{\Gamma_x: x\in L_{i,1} \}$.
Then $|E^*|\geq \frac{1}{4}|D|$. By \eqref{D-bound}, we have
 $$|E^*|\geq \frac{r-1}{2^{r+3}m(mpr)^{m-1}}|L_h|\geq \frac{1}{c}|E|.$$
By our discussion, $\phi$ is strongly rainbow on $\Gamma$ and hence $E^*$ is the
$r$-expansion of $\Gamma$. 
It is easy to check that $E^*, S$,  and $\Gamma$ satisfy the other requirements of the Lemma.

\medskip

{\bf Subcase 2.3} $|B'_i|\geq \frac{B^+}{2h}$ for some $1\leq i\leq h-1$.

\medskip

The arguments are similar in this subcase as in Subcase 2.3,
except that the proof of an analogous statement of Claim 2 is more delicate.
As in Subcase 2.2, we define $D$ and $D_x$ analogously
with $L_i$ being replaced by $L'_i$ in the definitions. 
Let $L'_{i,1}=\{x\in L'_i, x \mbox{ is heavy and is of type 1}\}$.
Let $L'_{i,2}=\{x\in L'_i, x \mbox{ is heavy and is of type 2}\}$.
As in Subcase 2.3, we may assume that

\begin{equation}
\sum_{x\in L'_{i,1}} |D_x|\geq \frac{1}{4}|D|.
\end{equation}
  
{\bf Claim 3.} Let $x,y\in L'_{i,1}, x\neq y$.
Let $e\in \Gamma_x$ and $f\in \Gamma_y$. Then $\phi(x)\cap \phi(y)=\emptyset$.

\medskip

{\it Proof of Claim 3.} We proceed like in the proof of Claim 2, with adjustments at the end.
Since $x,y$ are of type 1, we have $\delta(\Gamma_x)\geq p$ and $\delta(\Gamma_y)\geq p$. Suppose for contradiction that $\phi(e)\cap \phi(f)\neq 
\emptyset$. Let $v\in \phi(e)\cap \phi(f)$. Then
$\phi(e)\cap \phi(f)=\{v\}$. Suppose $e=ab$ and $f=a'b'$, where $a\in A_x, b\in C_x$
and $a'\in A_y, b'\in C_y$.  Let $\ell=2m-2-2h$. 
Since $\delta(\Gamma_y)\geq p=2mr>r\ell+2r$, by Lemma \ref{rainbow-path}, there
exists a path $P$ in $\Gamma_y$ of length $\ell$ starting at $b'$ that is strongly rainbow
under $\phi$ and such that $(\bigcup_{e'\in P} \phi(e'))\cap (\phi(e)\cup \phi(f))=\emptyset$.
Let $b''$ denote the other endpoint of $P$. Since $P$ has an even length, $b''\in C_y$.
Let $P^+$ denote the set of the edges of $E$ that correspond to the edges of $P\cup \{e,f\}$.
Then $P^+$ is linear path of length $2m-2h$ in $G$ where $a$ is an endpoint at one
and $b''$ is an endpoint at the other end. Furthermore, $V(P^+)\cap V(H)\subseteq
L_h$. Let $R$ be a monotone path in $H$ from $x$ to $a$. Then $R$ has length $h-i$ and
is internally disjoint from $P^+$. Since $b''\in C_y$, $y$ is a dominator of $b''$.
By definition, there exist $p$ internally disjoint $y,b''$-paths that intersect each main level
of $H$ at most once. In particular these paths have length $h-i+1$. 
Since $p=2mr>|V(P^+\cup R)|$, one of these paths, say $R'$, is internally disjoint from $P^+\cup R$.
Now $W=P^+\cup R\cup R'$ is  a linear $x,y$-path of length $2m-2i+1$ in $G$.
Let $e_x$ denote the edge of $W$ containing $x$ and $e_y$ the edge of $W$ containing $y$.
Each of $e_x, e_y$ intersects $L_i$ in exactly one vertex. Suppose $e_x\cap L_i=\{x^*\}$
and $e_y\cap L_i=\{y^*\}$. Then $V(W)\cap (\bigcup_{j=0}^i V(H_j))=\{x^*,y^*\}$.

By Lemma \ref{joining-path} there is an $x^*,y^*$-path $R_0$ of length $2t\leq 2i$ in
$\bigcup_{j=0}^i H_j$ such that $V(R_0)\cap L_i=\{x^*, y^*\}$.  If
$t=i$ then $W\cup R_0$ is a linear cycle in $G$ of length $2m+1$, a contradiction.
So suppose $t<i$. The idea now is to keep $R, R_0$ and $e_y$ and redefine $P$ and $R'$
to get a linear cycle of length $2m+1$.
Let $\ell=2m-2h+2(i-t)-3$. Note that $\ell>0$ and is odd.
Since $\delta(\Gamma_y)\geq p=2mr>r\ell+2r$, by Lemma \ref{rainbow-path}, there
exists a path $P$ in $\Gamma_y$ of length $\ell$ starting at $a'$ that is strongly rainbow
under $\phi$ and such that $(\bigcup_{e'\in P} \phi(e'))\cap (\phi(e)\cup \phi(f))=\emptyset$.
Let $b''$ denote the other endpoint of $P$. Since $\ell$ is odd, $b''\in C_y$.
Let $P^+$ denote the set of the edges of $E$ that correspond to the edges of $P\cup \{e,f\}$.
Then $P^+$ is linear path of length $2m-2h+2(i-t)-1$ in $G$ where $a$ is an endpoint at one
and $b''$ is an endpoint at the other end. Furthermore, $V(P^+)\cap V(H)\subseteq
L_h$.  As before, there are $p$ internally disjoint $y,b''$-paths of length $h-i+1$ hitting each main level at most once. Since $p=2mr>|V(P^+\cup R\cup e_y)|$, one of these paths, say $R'$,
is internally disjoint from $P^+\cup R\cup e_y$. It is also internally disjoint from $R_0$ 
by the definition of $R_0$. Now $P^+\cup R\cup R_0\cup \{e_y\}\cup R'$ is a linear cycle of
length $2m+1$ in $G$, a contradiction.
 \qed

Now by Claim 3 and earlier discussion,  $\forall x,y\in L'_{i,1}, x\neq y$,
we have $V(E_x)\cap V(E_y)=\emptyset$.
Let $E^*=\bigcup\{E_x: x\in L'_{i,1}\}$,
$S=\bigcup\{A_x: x\in L'_{i,1}\}$, and 
$\Gamma=\bigcup\{\Gamma_x: x\in L'_{i,1}\}$.
Then $|E^*|\geq \frac{1}{4}|D|$. By \eqref{D-bound}, we have
 $$|E^*|\geq \frac{r-1}{2^{r+3}m(mpr)^{m-1}}|E|\geq \frac{1}{c}|E|.$$
As in Subcase 2.2, it is easy to check that $E^*, S$,  and $\Gamma$ satisfy the four conditions of the Lemma.

\end{proof}

\begin{theorem} \label{odd-cycles}
Let $m,r$ be positive integers where $m\geq 2$ and $r\geq 3$.
There exist a positive real $c'_{m,r}$ and a positive integer $n_2$ such that
for all $n\geq n_2$ we have $ex_L(n,C^r_{2m+1})\leq c'_{m,r} n^{1+\frac{1}{m}}$.
\end{theorem}
\begin{proof}
We follow the steps in Theorem \ref{even-cycles}, using Lemma \ref{odd-expansion} in place of Lemma \ref{even-expansion}.  Let $p=2mr$. Let $c=2^{r+2}(mpr)^m$ as in Lemma \ref{odd-expansion}. 
Let $c'_{m,r}=2m^{r-1}c^m$. Choose $n_2$ such that $c'_{m,r}n_2^\frac{1}{m}\geq n_0$, where $n_0$ is given in Lemma \ref{split}. Let $G$ be an $n$-vertex linear $r$-graph with
at least $c'_{m,r}n^{1+\frac{1}{m}}$ edges, where $n\geq n_2$. Suppose that 
$G$ does not contain a copy of $C^r_{2m+1}$, we derive a contradiction.
By our assumption, $G$ has average degree at least $rc'_{m,r} n^{\frac{1}{m}}$.
By Lemma \ref{min-degree}, there exists a subgraph $G_0$ of $G$ with $\delta(G_0)\geq 
c'_{m,r}n^{\frac{1}{m}}$. Let $N=n(G_0)$. Then $N\geq c'_{m,r} n^\frac{1}{m}\geq n_0$ and $\delta(G')\geq c'_{m,r} N^{\frac{1}{m}}$. By Lemma \ref{split} (with $t=m$),
there exists a partition of $V(G')$ into $S_0,\ldots, S_{m-1}$
such that for each $u\in V(G')$ and $i\in \{0,\ldots, m-1\}$, we have $|L_{G'}(u)\cap G'[S_i]|\geq
\frac{c'_{m,r}}{2m^{r-1}} N^{\frac{1}{m}}=c^m N^{\frac{1}{m}}$.

Let $w$ be any vertex in $S_0$. Let $L_0=\{w\}$. Inside $G'$,
we will construct a leveled linear quasi tree $H$ of height $m$ 
rooted at $w$ with segments $H_0,\ldots, H_{m-1}$ and
main levels $L_0, L_1,\ldots, L_m$  such that $\forall i\in \{0,\ldots, m-1\}$
$V(H_i)\subseteq S_i$. (Note that this means $\forall i\in [m], L_i\subseteq S_{i-1}$). Furthermore, we will maintain that $\forall i\in [m]$, $|L_i|\geq N^{\frac{1}{m}}|L_{i-1}|$. This will imply that $|L_m|\geq N$, which is a contradiction.

We construct $H$ as follows.
Let $H_0$ consist  of the edges of $G'[S_0]$ containing $w$.
By our assumption, $|H_0|\geq c^m N^{\frac{1}{m}}\geq N^{\frac{1}{m}}$, by our definition of
$c$.  Let $L_1$ consists of a vertex from $e\setminus \{w\}$ for each $e\in H_0$.
We have $|L_1|=|H_0|\geq N^{\frac{1}{m}}|L_0|$.  In general, suppose $1\leq i\leq m-1$ and 
suppose we have defined $H_0,\ldots, H_{i-1}$ and $L_0, L_1,\ldots, L_i$ that
satisfy the requirements.
Let $E$ denote the set of edges in $G'$ that contain one vertex in $L_i\subseteq S_{i-1}$
and $r-1$ vertices in $S_i$. By the definition of the partition $(S_0,\ldots, S_{m-1})$, 
$|E|\geq c^m N^{\frac{1}{m}} |L_i|\geq c^i|L_i|$.
Since $C^r_{2m+1}\not\subseteq G'$, by Lemma \ref{odd-expansion},  there exists
a subset $E^*\subseteq E$ such that (1) $|E^*|\geq \frac{1}{c^i} |E|$,
(2) $E^*$ is the $r$-expansion of some bipartite $2$-graph $\Gamma$ with one part
in $L_i$ and the other part outside $\bigcup_{j=0}^{i-1} H_{i-1}$.
Now, let $H_i$ be the $r$-graph formed by $E^*$ and let
$L_i$ consist of one vertex from $e\setminus V(\Gamma)$ for each $e\in E^*$
(note that this implies that $|L_i|=|E^*|$).
Now, $\bigcup_{j=0}^i H_i$ is a leveled linear quasi-tree in $G'$ rooted at $w$ with heigh $i$
and main levels $L_0,L_1,\ldots, L_i$.  Furthermore, $|L_i|=|E^*|\geq \frac{1}{c^i}
|E|\geq \frac{c^m}{c^i} N^{\frac{1}{m}}|L_i|\geq N^{\frac{1}{m}}|L_i|$.
We can continue like this to construct $H$ and derive the desired contradiction.
\end{proof}

%%%%%%%%%%%%%%%%%%%%%%%%%%%%%%%%%%%%%%%%%%%%%%%%%%%

\section{Cycle-complete Ramsey numbers} \label{cycle-complete}

Given two $r$-graphs $G$ and $H$, the {\it Ramsey number} $R(G,H)$ is the smallest positive
integer $n$ such that in every coloring of the edges of $K^r_n$ using two colors red and blue
there exists either a red copy of $G$ or a blue copy of $H$.
As mentioned in the introduction, part of the motivitation behind our study of the
linear Tur\'an number of linear cycles comes from the study by Kostochka, Mubayi, and Vertra\"ete \cite{KMV} on
the hypergraph Ramsey number of a linear triangle versus a complete graph.
Their work is further inspired by the work of graph Ramsey number $R(C_3, K_t)$.
A celebrated result of Kim \cite{Kim} together with earlier upper bounds by 
Ajtai, Koml\'os, and Szemer\'edi \cite{AKS} shows that 
$$R(C_3, K_t)=\Theta(\frac{t^2}{\log t}), \mbox { as } t\to \infty.$$ 
Kostochka, Mubayi, and Verstra\"ete's main theorem \cite{KMV} is

\begin{theorem} {\bf \cite{KMV}}
There exist constants $a,b_r>0$ such that for all $t\geq 3$,
$$\frac{a t^{\frac{3}{2}}}{(\log t)^\frac{3}{4}} \leq R(C^3_3, K^3_t)\leq b_3 t^\frac{3}{2},$$
and for $r\geq 4$,
$$\frac{t^{\frac{3}{2}}}{(\log t)^{\frac{3}{4}+o(1)}}\leq R(C^r_3, K^r_t)
\leq b_r t^{\frac{3}{2}}.$$
\end{theorem}
In addition, they showed
\begin{theorem}{\bf \cite{KMV}}  \label{KMV-lower}
For fixed $r,k\geq 3$, $$R(C_k, K^r_t)=\Omega^*(t^{1+\frac{1}{3k-1}}),
\mbox{ as } t\to \infty.$$
There exists a constant $c_r>0$ such that
$$R(C_5, K^r_t)\geq c_r(\frac{t}{\ln t})^\frac{5}{4}, \mbox{ as } t\to\infty.$$
\end{theorem}
Here the authors use $f=O^*(g)$ to denote that for some constant $c>0, f(t)=O((\ln t)^c g(t))$,
and $f=\Omega^*(g)$ is equivalent to $g=O^*(f)$.  
The key point of Theorem \ref{KMV-lower} is that the exponent $1+\frac{1}{3k-1}$ of $t$ is 
bounded away from $1$ by a constant independent of $r$.
The authors made the following conjecture.

\begin{conjecture} {\bf \cite{KMV}}
For all fixed $r\geq 3$, $R(C_3, K^r_t)=o(t^{3/2})$ and $R(C_5, K^r_t)=O(t^{5/4})$, as $t\to \infty$.
\end{conjecture}

Using our bounds on the linear Tur\'an numbers, we can quickly derive nontrivial upper bounds
on $R(C^r_\ell, K^r_t)$ for all $r, \ell\geq 3$.
Before getting into that, we give some recount on the cycle-complete Ramsey 
numbers of graphs. As mentioned above, 
the behavior of $R(C_3, K_t)$ is now quite well understood, particularly with
the recent deep works in \cite{BK1}, \cite{FGM}.  For longer cycles, the 
best known upper bounds are $R(C_{2m},K_t)=O ((\frac{t}{\ln t})^{\frac{m}{m-1}})$
due to  Caro et al \cite{caro} and $R(C_{2m+1}, K_t)=O(\frac{t^{\frac{m+1}{m}}}{(\ln t)^{1/m}})$, due
to Sudakov \cite{sudakov} and Li and Zang \cite{Li-Zang}. The best known lower bound is
$R(C_\ell, K_t)=\Omega(\frac{t^{\frac{\ell-1}{\ell-2}}}{\ln t})$, due to Bohman and Keevash \cite{BK}. 
 
We now obtain some upper bounds on $R(C^r_\ell, K^r_t)$ 
using linear Tur\'an numbers and a reduction process via the well-known
sunflower lemma. A sunflower (or $\Delta$-system) $\cF$  with core $C$  is a collection of 
distinct sets $A_1,\ldots, A_p$ such that $\forall i,j\in [p]$ we have $A_i\cap A_j=C$. We call the $A_i$'s
{\it members} of the sunflower. If a sunflower has $p$ members and the core has size $a$,
then we call it a $(a,p)$-sunflower.
Note that the core is allowed to be empty and hence a matching  is considered to be a sunflower. 

\begin{lemma} {\bf (Sunflower Lemma \cite{ER})} \label{sunflower}
If $\cF$ is a collection of sets of size at most $k$ and $|\cF|\geq k!(p-1)^k$, then
$\cF$ contains a sunflower with $p$ members.
\end{lemma}

Partly following the approach in \cite{KMV}, we consider 
non-uniform hypergraphs, but will disallow singletons as edges. 
Recall that a {\it linear cycle} of length $\ell$ is a list of sets $A_1,\ldots, A_\ell$ such that
$|A_i\cap A_{i+1}|=1$ for $i=1,\ldots,\ell-1$, $|A_\ell\cap A_1|=1$ and $A_i\cap A_j=
\emptyset$ for all other pairs $i,j$, $i\neq j$.
A set $S$ in a hypergraph $G$ is an {\it independent set} in $G$ if no edge of $G$ is
contained in $S$. Let $\alpha(G)$ denote the maximum size of an independent set in $G$.
The next lemma is similar to the ones in \cite{KMV}, except that  we use the sunflower lemma. A hypergraph is {\it simple} if no edge contains another.

\begin{lemma} \label{contraction}
Let $m,r\geq 2$ be integers. 
Let $G$ be a hypergraph whose edges have sizes between $2$ and $r$.
Suppose $G$ does not contain a linear cycle of length $\ell$.
Then there exists a simple hypergraph $G'$ on $V(G)$ whose edges have sizes
between $2$ and $r$ such that $G'$ contains no linear cycle of length $\ell$,
$G'$ contains no $(a,r\ell)$-sunflower for any $a\geq 2$,
and $\alpha(G')\leq \alpha(G)$.
\end{lemma}
\begin{proof}
We iterate the following process. 
Let $\cF$ be an $(a,r\ell)$-sunflower in $G$ with core $C$, where $|C|=a\geq 2$.
Let $G_1$ be obtained from $G$ by replacing some edge $e$ in $\cF$ with $C$.
If $G_1$ contains a linear cycle $L$ of length $\ell$, then $L$ must use $C$ as an edge. Since $L$ contains at most $r\ell$
vertices and $C$ is the core of a sunflower $\cF$ with $r\ell$ members, we can find
some edge $e'$ in $\cF$ such that $e'\setminus C$ is disjoint from $V(L)$. Now if we
replace $C$ with $e'$ in $L$, we obtain a linear cycle of length $\ell$ in $G$, a contradiction.
So, $G_1$ has no linear cycle of length $\ell$.
Clearly, any independent set $S$ in $G$ is also an independent set in $G_1$. So $\alpha(G_1)\leq \alpha(G)$. We now replace $G$ with $G_1$ and repeat this process until there is no longer  an $(a,r\ell)$-sunflower for some $a\geq 2$. 
The process must end since the total edge-size decreases at each step. Denote the final graph by $G'$. If $G'$ is not simple then we make it simple by removing edges that contain other edges. This
cannot create a linear cycle of length $\ell$, or a new sunflower, or increase the independence number. Then $G'$ satisfies the claim.
\end{proof}

A hypergraph $G$ is {\it $(2,q)$-linear} if no pair of vertices is contained in $q$ or more
edges of $G$.

\begin{lemma} \label{2q-linear}
Let $a,p,r\geq 2$ be integers.
Let $G$ be a simple hypergraph whose edges have sizes between $2$ and $r$
and contains no $(a,p)$-sunflower for any $a\geq 2$. Then $G$ is $(2,q)$-linear, where $q=r!(p-1)^r$.
\end{lemma} 
\begin{proof}
Otherwise some pair $\{a,b\}$ 
is contained in a set $H$ of least $q$ edges of $G$. Let $H'=\{e\setminus\{a,b\}: e\in H\}$.
Since $H\subseteq G$ is simple $|H'|=|H|\geq q=r!(p-1)^r$. By Lemma \ref{sunflower}, 
$H'$ contains a sunflower $\cF$ with
$p$ members. Now, adding $\{a,b\}$ to each member of $\cF$ yields an $(a,p)$-sunflower
in $G$, where $a\geq 2$,
contradicting our assumption about $G$. 
\end{proof}

\begin{lemma} \label{linear-subgraph} 
Let $r\geq 2,q\geq 1$ be integers. Let $G$ be a hypergraph whose edges have sizes
between $2$ and $r$.
Suppose $G$ is $(2,q)$-linear. Then $G$ contains a linear subgraph $G'$ with
$|G'|\geq \frac{2}{qr^2} |G|$.
\end{lemma}
\begin{proof}
By our assumption, each edge $e$ of $G$ shares 
a pair of vertices with at most $\binom{r}{2}(q-1)$ other edges.
Let $H$ be a graph whose vertices are the edges of $G$ such that
two vertices $u,v$ are adjacent in $H$ if the corresponding edges in $G$
share a pair of vertices. Then $\Delta(H)<\binom{r}{2}q-1$.
Hence $H$ contains an independent set $S$ of size at least $\frac{|V(H)|}{\Delta(H)+1}
\geq \frac{2|V(H)|}{qr^2}$. Let $G'$ be the subgraph of $G$ whose edges correspond to $S$.
Then $G'$ is a linear subgraph of $G$ with $|G'|\geq \frac{2}{qr^2}|G|$.
\end{proof}

\begin{lemma} \label{local-sparse}
Let $H$ be a linear hypergraph whose edges have sizes between $2$ and $r$.
Suppose $H$ does not contain a linear cycle of length $\ell$.  Let $D=\partial_2(H)$.
Let $v$ be any vertex in $V(D)=V(H)$. Then $|D[N_H(v)]|\leq r^{r+4}\ell |N_H(v)|$.
\end{lemma}
\begin{proof}
Since $H$ is linear, the link graph $\cL_H(x)$ consists of disjoint edges each of size at most
$r-1$. Let $U=V(\cL_H(v))=N_H(v)$. The edges of $\cL_H(x)$ form a partition of $U$ into
parts of size at most $r-1$ (with each part being an edge of $\cL_H(x)$).
Also since $H$ is linear no edge of $H[U]$ contains more than one vertex from any of those parts. Let us randomly and independently pick one vertex from each part, and call the resulting set $S$. For each edge in $H[U]$ the probability of it being in $H[S]$ is
at least $(\frac{1}{r-1})^r$. So there is a choice of $S$ for which
$H[S]\geq \frac{1}{(r-1)^r} |H[U]|$. If $H[S]$ has average degree at least $r^2\ell$, then
it contains a subgraph $H'$ with minium degree at least $r\ell$ and since $H'$
is linear, one can easily find a linear path $P$ of length $\ell-2$ say with endpoint $a$ and $b$.
Let $e_a$ be the edge of $H$ that contains $\{x,a\}$ and $e_b$ the edge of $H$ that contains $\{x,b\}$. 
Then $e_a\cap S=\{a\}, e_b\cap S=\{b\}$. In particular, we see that
$P\cup \{e_a, e_b\}$ is a linear cycle of length $\ell$, a contradiction. So $H[S]$ has
average degree less than $r^2\ell$. So, $|H[U]|\leq (r-1)^r |H[S]|<r^r\frac{r^2}{2}\ell|S|<
r^{r+2}\ell|U|$. So  $|D[U]|\leq \binom{r}{2}|H[U]|<r^{r+4}\ell |U|$.
\end{proof}

We need the following lemma due to Alon \cite{Alon-sparse}. 
The version stated below is implicit in the proof of Proposition 2.1 in \cite{Alon-sparse}. Alternatively, one could also apply  \cite{AKS-sparse}. Logarithms below are in base $2$.

\begin{lemma} {\bf \cite{Alon-sparse}} \label{alon-sparse}
Let $G$ be an $n$-vertex 
graph with maximum degree at most $d\geq 1$, in which for any vertex $v$,
$G[N(v)]$ contains an independent set of size at least $\frac{|N(v)|}{p}$. Then
$\alpha(G)\geq \frac{n\log d}{160 d \log (p+1)}$.
\end{lemma} 

\begin{theorem}
Let $m,r$ be integers where $m\geq 2$ and $r\geq 3$.
There exists a constant $a_{m,r}$, depending on $m$ and $r$
such that $R(C^r_{2m}, K^r_t)\leq a_{m,r}(\frac{t}{\ln t})^\frac{m}{m-1}$.
\end{theorem}
 \begin{proof}
The definition of $a_{m,r}$ depends on various constants we defined earlier and will be implit in our proof. 
Let $n\geq a_{m,r} (\frac{t}{\ln t})^\frac{m}{m-1}$. 
By choosing $a_{m,r}$ to be large enough, we may assume that $n\geq n_1$, where $n_1$ is given in Theorem \ref{even-cycles}. It suffices to show that if $G$ is
an $n$-vertex $r$-graph that does not contain $C^r_{2m}$ then $G$ contains an independent 
set of size at least $t$. Let such $G$ be given.  By Lemma \ref{contraction}, there exists
a simple 
hypergraph $G'$ with $V(G')=V(G)$ such that $\alpha(G')\leq \alpha(G)$, $G'$ contains no
linear cycle of length $2m$, and that $G'$ contains no $(a,2mr)$-sunflower for 
any $a\geq 2$.
By Lemma \ref{2q-linear}, $G'$ is $(2,q)$-linear, where $q=r!(2mr-1)^r$.
By Lemma \ref{linear-subgraph}, $G'$ contains a linear subgraph with
$|G''|\geq c_1|G'|$, where $c_1$ is a positive constant depending on $m$ and $r$. Clearly, $G''$ contains no linear cycle of length $2m$. Applying the $O(n^{1+\frac{1}{m}})$ bound
\cite{BS}  on $ex(n,C_{2m})$
and Theorem \ref{even-cycles},  by considering edges of various sizes, we have $|G''|\leq c_2 n^{1+\frac{1}{m}}$, for some constants $c_2$, depending on $m$ and $r$. Hence $|G'|\leq c_3 n^{1+\frac{1}{m}}$ for some constant $c_3$, depending on $m$ and $r$. So $G'$ has average degree 
at most $rc_3n^{\frac{1}{m}}$. Clearly, at most $n/2$ vertices in $G'$ can have degree at least
$2rc_3n^{\frac{1}{m}}$. Let $H$ be the subgraph of $G'$ induced by vertices of degree at most
$2rc_3 n^\frac{1}{m}$. Then $|V(H)|\geq \frac{n}{2}$ and $\Delta(H)\leq 2rc_3n^\frac{1}{m}$.

Let $D=\partial_2(H)$. Then $\Delta(D) \leq 2r^2 c_3n^{\frac{1}{m}}$.
Note that for each vertex $v$ we have $N_D(v)=N_H(v)$, which we will denote by $N(v)$.
Since $H$ does not contain a linear cycle of length $2m$, by Lemma \ref{local-sparse},
for each vertex $v$ in $V(H)=V(D)$, we have $|D[N(v)]|\leq 2mr^{r+4}|N(v)|$. So $D[N(v)]$ has
average degree at most $4mr^{r+4}$. By Caro and Wei \cite{Caro, Wei},
$D[N(v)]$ contains an independent set of size at least $\frac{|N(v)|}{4mr^{r+4}+1}$.
By Lemma \ref{alon-sparse}, with $d= 2r^2 c_3n^{\frac{1}{m}}$, $\alpha(D)\geq 
c_5 \frac{|V(D)|\ln n}{n^\frac{1}{m}}\geq \frac{c_5}{2} 
n^{\frac{m-1}{m}}\ln n$, for some positive constant $c_5$,
depending on $m$ and $r$. Since $n\geq a_{m,r} (\frac{t}{\ln t})^\frac{m}{m-1}$,  by choosing $a_{m,r}$ to be large enough, we can ensure $\alpha(D)\geq t$. Certainly any indepdent set in $D$ is
also an independent set in $G'$. Hence $\alpha(G')\geq t$ and $\alpha(G)\geq \alpha(G')=t$.
\end{proof}

For odd cycle-complete Ramsey numbers, we need some more definitions and a lemma.
Let $H$ be a hyergraph whose vertices are ordered by a total order $\pi$. Let $P$
be a linear path of length $\ell$, that is, $P$ consists of a list of edges $e_1,\ldots, e_\ell$
such that $|e_i\cap e_{i+1}|=1$ for each $i\in [\ell-1]$ and $e_i\cap e_j=\emptyset$ whenever
$|i-j|>1$. For each $i\in [\ell-1]$, let $e_i\cap e_{i+1}=\{x_i\}$.
We say that $P$ is an {\it increasing linear path} under $\pi$  if for all $v\in e_1\setminus \{x_1\}, \pi(v)<\pi(x_1)$, $\forall v\in e_\ell\setminus x_{\ell-1}, \pi(x_{\ell-1})<\pi(v)$, and for each $i= 2,\ldots,\ell-1$ and $v\in e_i\setminus \{x_{i-1},x_i\}$, we have $\pi(x_{i-1})<\pi(v)<\pi(x_i)$.
If $P$ is an increasing linear path and $v$ is the
largest vertex on $P$ under $\pi$, then we say that $P$ {\it ends} at $v$.

\begin{lemma} \label{increasing-path}
Let $H$ be a hypergraph and $\pi$ a total order on $V(H)$.
If $H$ does not contain an increasing linear path of
length $\ell$, then $V(H)$ can be partitioned into $\ell$ independent sets.
\end{lemma}
\begin{proof}
For each $i=0,\ldots \ell-1$, let $S_i$ denote the set of vertices $v$ such that 
the longest increasing linear path in $H$ that ends at $v$ has length $i$. Then $S_0,\ldots, S_{\ell-1}$
partition $V(H)$. Suppose for some $i\in \{0,\ldots, \ell-1\}$, $S_i$ contains an edge $e$.
Let $v$ and $v'$ be the vertices in $e$ that are smallest and largest under $\pi$, respectively. By definition, $H$ contains an increasing linear path $P$ of length $i$ that ends at $v$. Now $P\cup e$ is an increasing path of length $i+1$ that ends at $v'$,
contradicting $v'\in S_i$. Hence for each $i$, $S_i$ contains no edge of $H$ and hence is an independent set in $H$.
\end{proof}

The following lemma is a variant of Theorem 1 in \cite{EFRS}. The proof is similar. 
\begin{lemma}  \label{odd-cycle-levels}
Let $H$ be a  hypergraph whose edges have sizes between $2$ and $r$. Suppose
$H$ does not contain a linear cycle of length $2m+1$. Let $H^*$ be the subgraph of $H$ consisting of all the edges of size $2$ in $H$. 
Let $v\in V(G)$. For each $i$, let $S_i$ be the set of vertices in $H^*$ that
are at distance $i$ from $v$. Then for each $i\leq m$, $H[S_i]$ contains an independent
set of size at least $\frac{|S_i|}{2m-1}$.
\end{lemma}
\begin{proof}
Grow a breadth-first search tree $T$ in $H^*$ from $v$. So the levels of $T$ are precisely the distance
classes from $v$ in $H^*$. For each $i\geq 1$, define a linear order $\pi_i$ of $S_i$ as follows.
Let $\pi_1$ be an arbitrary linear order on $S_1$. For each $i\geq 2$, let $\pi_i$ be a linear
order on $S_i$ obtained by listing the children of the first vertex in $\pi_{i-1}$, followed
by the children of the second vertex in $\pi_{i-1}$, and etc.
For each $1\leq i\leq m$, we claim that $H[S_i]$ contains no increasing linear path of length $2m-1$.
Otherwise, fix an $i$ for which $H[S_i]$ contains an increasing linear path $P$ of length $2m-1$ with 
edges $e_1,e_2,\ldots, e_{2m-1}$ in order.  Let $x_1$ be the least vertex in $e_1$ under $\pi_i$. Let $x_{2m}$ be the largest vertex in $e_{2m-1}$ under $\pi_i$. For each $k\in\{2,\ldots, 2m-1\}$, let $e_{k-1}\cap e_k=\{x_k\}$. Then $x_1<x_2<\ldots<x_{2m}$ in $\pi_i$.
Let $w$ be a closest common anchester of $x_1,\ldots, x_{2m}$ in $T$. Suppose $w\in S_j$, where $j<i$. Let $k$ be the smallest positive integer such that $x_k$ and $x_{k+1}$ are under
different children of $w$. Such $k$ exists by our choice of $w$. By our ordering on each level,
the anchesters of $x_1,\ldots, x_k$ in $S_j$ precede anchesters of $x_{k+1},\ldots, x_{2m}$ 
in $S_j$ under $\pi_j$. Hence for any $a\in [k], b\in [2m]\setminus [k]$, the unique $x_a,x_b$-path $Q_{a,b}$ in $T$ must pass through $w$ and has length $2(i-j)$.
Based on the value of $k$, we can find $a\in [k], b\in [2m]\setminus [k]$
such that $b-a=2m+1-2(i-j)$. Now $Q_{a,b}\cup\{e_a,e_{a+1},\cdots, e_{b-1}, e_b\}$ is a linear cycle of length $2m+1$ in $H$, a contradiction. 
Hence $H[S_i]$ contains no increasing linear path of length $2m-1$. By  Lemma \ref{increasing-path}, $H[S_i]$ contains an independent set of size at least
$\frac{|S_i|}{2m-1}$.
\end{proof}

\begin{theorem}
Let $m,r$ be positive integers where $m\geq 2, r\geq 3$. There exists a positive constant $b_{m,r}$,
depending on $r$ and $m$, such that $R(C^r_{2m+1}, K^r_t)\leq b_{m,r} t^{\frac{m}{m-1}}$.
\end{theorem}
\begin{proof}
Our choice of $b_{m,r}$ will depend on other constants defined earlier and will be implicit in the proof. 
Let $n\geq b_{m,r}t^{\frac{m}{m-1}}$. 
By choosing $b_{m,r}$ to be large enough, we may assume that $n\geq n_2$, where $n_2$ is specified in Theorem \ref{odd-cycles}.
Let $G$ be any $n$-vertex $r$-graph on $n$ vertices
not containing a copy of $C^r_{2m+1}$. We show that $G$ 
contains an independent set of size at least $t$.

By Lemma \ref{contraction}, there exists
a simple hypergraph $G'$ on $V(G)$ whose edges have sizes between $2$ and $r$ 
such that $\alpha(G')\leq \alpha(G)$, $G'$ contains no linear cycle of length $2m+1$, and that $G'$ contains no $(a,(2m+1)r)$-sunflower for any $a\geq 2$. By Lemma \ref{2q-linear}, $G'$ is $(2,q)$-linear where $q=r![(2m+1)r-1]^r$.
For each $j=3,\ldots, r$, let $G_j$ denote the subgraph of $G'$ consisting of edges of
size $j$. Let $G''=\bigcup_{j=3}^r G_j$. Then $G''$ is $(2,q)$-linear. By Lemma \ref{linear-subgraph}, $G''$ contains a linear subgraph $G^*$ with $|G^*|\geq \frac{2}{qr^2}|G''|$.
By Theorem \ref{odd-cycles}, $|G^*|\leq c'_1 n^{1+\frac{1}{m}}$ for some positive constant
$c'_1$ depending on $m$ and $r$. Hence $|G''|\leq c'_2 n^{1+\frac{1}{m}}$ for
some positive constant $c'_2$ depending on $m$ and $r$.
The number of vertices of $G''$ of degree at least $2rc'_2 n^{\frac{1}{m}}$ is at most $n/2$. Let $U$ be the set of vertices of degree at most $2rc'_2n^\frac{1}{m}$
in $G''$. Then $|U|\geq \frac{n}{2}$.
Let $H=G'[U]$. 
Let $H^*$ be the subgraph of $H$ consisting of edges of size $2$.
Let $H'$ be subgraph of $H$ consisting of edges of size $3$ or more. By our definition of $H$,
$\Delta(H')\leq 2rc'_2n^\frac{1}{m}$.  We obtain a large independent set
$W$ in $H$ as follows. Initially set $W=\emptyset$.
Let $v$ be any vertex in $H$ and for each $i\geq 0$ let $S_i$ denote the set of vertices
at distance $i$ from $v$ in $H^*$.   Let $0\leq k\leq m-1$ be
the smallest  integer such that $\frac{|S_{i+1}|}{|S_i|}\leq n^{\frac{1}{m}}$. Such $k$ exists since otherwise we would have $|S_m|>n$, a contradiction.
Since $H$ contains no linear cycle of length $2m+1$, by Lemma \ref{odd-cycle-levels}, 
$H[S_k]$ contains an independent set $S'$ of size at least $\frac{|S_k|}{2m-1}$.
Let $\widetilde{S}=S_{k-1}\cup S_k\cup S_{k+1}$ (or $\widetilde{S}=S_0\cup S_1$, if $k=0$). Then
the neighbors in $H^*$ of vertices in $S'$ lie in $\widetilde{S}$.
By our choice of $k$, $|\widetilde{S}|<(n^\frac{1}{m}+2)|S_k|<(2m-1)(n^\frac{1}{m}+2)|S'|<3mn^\frac{1}{m}|S'|$. Let $Z$ be a set of vertices in $H$ obtained by
picking a vertex in $e\setminus \widetilde{S}$, if exists, for each edge $e$ 
in $H'$ that contains a vertex in $S'$. Since $\Delta(H')\leq 2rc'_2n^\frac{1}{m}$,
we have $|Z|\leq 8rc'_2 n^\frac{1}{m}|S'|$. By our discussion above, $|\widetilde{S}\cup Z|\leq
c'_3 n^\frac{1}{m}|S'|$ for some positive constant $c'_3$ depending on $m$ and $r$.
We add $S'$ to $U$ and delete $\widetilde{S}\cup Z$ from $H$ and iterate the process
until we run out of vertices. By design, the final $W$ is an independent set in $H$
that has size at least $\frac{n/2}{c'_3 n^\frac{1}{m}}\geq \frac{n^\frac{m-1}{m}}{2c'_3}$.
Since $n\geq b_{m,r} t^\frac{m}{m-1}$, by choosing $b_{m,r}$ to be large enough, we can ensure
$\alpha(H)\geq t$. Since $H=G'[U]$, we have $\alpha(G)\geq \alpha(G')\geq t$. 
\end{proof}
%%%%%%%%%%%%%%%%%%%%%%%%%%%%%%%%%%%%%%%%%%%%%%%%%%%

\section{Concluding Remarks}

Our main objective in this paper is to establish an $O(n^{1+\fl{\frac{2}{\ell}}})$ bound
on $ex_L(n,C^r_\ell)$. We chose constants
$c_{r,\ell}$ and $c'_{r,\ell}$ in Theorem \ref{even-cycles} and Theorem \ref {odd-cycles}. 
larger than necessary in order to simplify our presentation. 
It is possible that like in the graph  case one could find a constant $c_r$, depending on $r$, such that $ex_L(n,C^r_{2m})\leq c_r m n^{1+\frac{1}{m}}$. It will be interesting
to see whether that indeed is the case.

The study of $ex_L(n,C^3_{2m})$ has a natural connection to the so-called {\it rainbow
Tur\'an number} $ex^*(n,C_{2m})$ of a cycle of length $2m$, which denotes the maximum
number of edges in an $n$-vertex 
graph that admits a proper edge-coloring that contains no cycle of
length $2m$ all of whose edges have different colors. The main conjecture from \cite{rainbow-Turan} is that $ex^*(n,C_{2m})=O(n^{1+\frac{1}{m}})$, which remains open except for $C_4$ and $C_6$. See Das, Lee, and Sudakov \cite{rainbow-cycle} for some recent progress on the problem.
Interestingly, there it is not too hard to obtain an $\Omega(n^{1+\frac{1}{m}})$ lower bound
on $ex^*(n,C_{2m})$ through an explicit construction 
using $B^*_k$-sets. 
Here, the difficulty in finding a good lower bound on $ex_L(n,C^r_\ell)$  for $r\geq 3$
is similar to that for $ex(n,C_\ell)$ for even cycles $C_\ell$.
Verstra\"ete \cite{Jacques-communication} observed that by taking
a random subgraph of a Steiner triple system one can show that
$ex(n,C^3_\ell)\geq \Omega(n^{1+\frac{1}{\ell-1}})$. Similarly,
by taking a random subgraph $H$ of a linear $n$-vertex $r$-graph $G$ with $(1-o(1))
\frac{\binom{n}{2}}{\binom{r}{2}}$ edges (such $G$ exists by the well-known packing result of R\"odl \cite {rodl}) and the usual deletion argument, one can show that
\begin{proposition} \label{Turan-lower-bound}
For all integers $r,\ell\geq 3$, $\exists$ a constant $c''_{r,\ell}>0$ such that $ex_L(n,C^r_\ell)>c''_{r,\ell} n^{1+\frac{1}{\ell-1}}$.
\end{proposition} 

Using generalized Sidon sets such as the ones considered in \cite{Ruzsa}
and \cite{LV}, it is conceivable that one can obtain a similar (or better) 
constructive lower bound on $ex_L(n,C^3_\ell)$ (and maybe also for all
$r\geq 3$.) This is an area worth some exploration.

Our Ramsey bounds on $R(C^r_\ell, K^r_t)$ are similar to those for graphs. However, as 
speculated in \cite{KMV}, for $r\geq 3$ perhaps $R(C^r_\ell, K^r_t)=\Theta^*(t^\frac{\ell}{\ell-1})$ holds, where $O^*$ and $\Omega^*$ are defined in Section 7.
It will be interesting to further sharpen our bounds on $R(C^r_\ell, K^r_t)$. By anaylzing the proof of Theorem \ref{odd-cycles}, together with Lemma \ref{local-sparse}
and Lemma \ref{alon-sparse}, one might be able to improve our bound on $R(C^r_{2m+1}, K^r_t)$ by a factor of $(\ln t)^c$. On the other hand, perhaps a more substantial improvement
is possible.

As in \cite{KMV},  let $RL(C^r_\ell, K^r_t)$ denote
the smallest $n$ such that every linear $r$-graph not containg
$C^r_\ell$ has an independent set of size $t$. 
Using our linear Tu\'an bounds and the usual random sampling arugment, 
one readily obtains
$RL(C^r_\ell, K^r_t)=O(t^\frac{\ell}{\ell-1})$.

%%%%%%%%%%%%%%%%%%%%%%%%%%%%%%%%%%%%%%%%%%%%%%%%%%%
\section{Acknowledgment} The authors would like to thank Kostochka, Mubayi, and Verstra\"ete for informative communications and Verstra\"ete for fruitful and stimulating discussions.

%%%%%%%%%%%%%%%%%%%%%%%%%%%%%%%%%%%%%%%%%%%%%%%%%%

\end{document}